\documentclass[12pt,letterpaper]{amsart}
\usepackage{amsfonts,amsmath,amssymb,amscd}
\usepackage{cite}
\usepackage{amsthm}
\usepackage{empheq}
\usepackage{tikz,graphics}
\usetikzlibrary{shapes,backgrounds,calc}
\usepackage{latexsym}
\usepackage{multirow,multicol}
\usepackage{bigstrut}
\usepackage{verbatim,enumerate}
\usepackage{array}
\usepackage[colorlinks=true, linkcolor=blue, citecolor=blue, urlcolor=black]{hyperref}
\usepackage{url}
\usepackage{bigstrut}
\usepackage{longtable}

\newtheorem{theorem}{Theorem}[section]
\newtheorem{lemma}[theorem]{Lemma}
\newtheorem{prop}[theorem]{Proposition}
\newtheorem{corollary}[theorem]{Corollary}
\newtheorem{defn}{Definition}

\newtheorem{example}{Example}
\numberwithin{equation}{section}

\newcommand{\Q}{\mathbb{Q}}
\newcommand{\N}{\mathbb{N}}
\newcommand{\Z}{\mathbb{Z}}

\renewcommand{\a}{\alpha}

\begin{document}
	\title[Newton polygon over composition]{Behaviour of Newton polygon over polynomial composition}
	\author[Jakhar]{Anuj Jakhar}
	\author[Srinivas]{Srinivas Kotyada}
	\author[Laishram]{Shanta Laishram}
	\author[Yadav]{Prabhakar Yadav}
	
	\address[Jakhar]{Department of Mathematics, Indian Institute of Technology (IIT) Madras}
	\email{anujjakhar@iitm.ac.in;} 
	
	\address[Kotyada]{Institute of Mathematical Sciences, CI of Homi Bhabha National Institute, CIT Campus, Taramani, Chennai 600 113, India}
	\email{srini@imsc.res.in; srinivas.kotyada@gmail.com} 
	
	\address[Laishram, Yadav]{Stat-Math Unit, Indian Statistical Institute \\
		7 S. J. S. Sansanwal Marg, New Delhi, 110016, India}
	\email[Laishram]{shanta@isid.ac.in;} 
	\email[Yadav]{pkyadav914@gmail.com; yadavprabhakar096@gmail.com}
	
	\dedicatory{}

	\begin{abstract}
	
In this paper, we study the structure of Newton polygons for compositions of polynomials over the rationals. We establish sufficient conditions under which the successive vertices of the Newton polygon of the composition \( g(f^n(x)) \) with respect to a prime \( p \) can be explicitly described in terms of the Newton polygon of the polynomial \( g(x) \). Our results provide deeper insights into how the Newton polygon of a polynomial evolves under iteration and composition, with applications to the study of dynamical irreducibility, eventual stability, non-monogenity of tower of number fields, etc.
	\end{abstract}

	\maketitle
	\pagenumbering{arabic}
	\pagestyle{headings}

\section{Introduction and Statement of Results}

Newton polygons, named after Isaac Newton, have proven to be an indispensable tool in the study of polynomials. They have numerous applications, including analyzing the factors of a polynomial, studying the Galois group of a polynomial, computing the discriminant of a number field, examining the monogenity and integral basis of number fields, and understanding the splitting of a rational prime in a number field, etc. These applications have attracted significant attention in recent years (cf. \cite{coleman}, \cite{FMN12} \cite{filaseta bessel}, \cite{FKT12}, \cite{FT02}, \cite{GMN12}, \cite{G2}, \cite{J1}--\cite{JLS18}, \cite{KK12}, \cite{LY24}--\cite{26}). 

Newton initially introduced polygons to study complex curves of two variables, which eventually led to the development of what is now known as the Puiseux series of a curve (cf. \cite{7} for details). This method also applies to polynomials in one variable by looking their various $g$-adic expansions. A significant theory in this regard was developed by Ore \cite{10} nearly a century ago. In his 1923 Ph.D. thesis and a series of subsequent papers, Ore extended the arithmetic applications of Newton polygons \cite{30}, \cite{10}. 

Let \( f(x) \in \mathbb{Z}[x] \) be a monic irreducible polynomial, \( K \) be the number field generated by a root \( \theta \) of \( f(x) \), and \( p \) be a prime number. Under a certain \( p \)-regularity condition, Ore provided a constructive method to determine the prime ideal decomposition of \( p \) in \( K \), the \( p \)-adic valuation of the index of the subgroup \( \mathbb{Z}[\theta] \) in the ring of algebraic integers of \( K \), and the factorization of \( f \) over the field \( \mathbb{Q}_p \) of \( p \)-adic numbers. The prime ideals are parameterized in terms of combinatorial data associated with different \( \phi \)-Newton polygons, where the polynomials \( \phi(x) \) are monic lifts to \( \mathbb{Z}[x] \) of the distinct irreducible factors of \( f(x) \) modulo \( p \). 

Ore's work has gained renewed interest in recent years (e.g., see \cite{CMS2000}, \cite{25}, \cite{26}). In \cite{Fadil21}, \cite{GMN12}, \cite{G2}, Guardia, Montes, and Nart extended Ore's ideas and developed the theory of higher-order Newton polygons as a powerful tool for computing discriminants, prime ideal decomposition, and integral bases of number fields. In 2019, Kim and Miller \cite{KM19} employed this tool to compute an integral basis and find integral elements of small prime power norm, ultimately establishing an upper bound for the class number. Using further algebraic arguments, they deduced that the class number is \( 1 \). Notably, computing the class number without assuming the Generalized Riemann Hypothesis (GRH) remains a highly challenging problem. 


To begin with, we recall the notion of Newton polygon. Let \( p \) be a prime number and let \( \nu_p \) denote the \( p \)-adic valuation on \( \mathbb{Q} \). Consider a polynomial \( g(x) = \sum_{i=0}^{e} a_i x^i \in \mathbb{Q}[x] \) with \( g(0) \neq 0 \). The \textit{Newton polygon} of \( g \) with respect to \( p \), denoted by \( NP_p(g) \), is the polygonal path formed by the lower edges along the convex hull in \( \mathbb{R}^2 \) of the points \( (e-i, \nu_p(a_i)) \) for \( 0 \leq i \leq e \). 

In 1906, Dumas \cite{duma} considered Newton polygons with respect to a prime $p$ of two polynomials $f(x), g(x) \in \Q[x]$ and gave a beautiful idea to construct the Newton polygon of the polynomial obtained by the product $f(x)g(x)$. This result helped in obtaining various results on the irreducibility of the polynomials. Further, with such a result, one can easily predict the Newton polygon structure of the polynomial $f(x)^n$ (\textit{\( n \)-times product}) for any polynomials $f(x) \in \Q[x]$. Dumas proved the following fundamental result:


\begin{theorem}\cite[Dumas' Theorem]{duma} \label{dumas thm}
	Let $g(x)$ and $h(x)$ be in $\mathbb{Z}[x]$ with $g(0)h(0) \neq 0$, and let $p$ be a prime. Let $k$ be a  non-negative integer such that $p^k$ divides the leading coefficient of $g(x)h(x)$ but $p^{k+1}$ does not. Then the edges of the Newton polygon for $g(x)h(x)$ with respect to $p$ can be formed by constructing a polygonal path begining at $(0,k)$ and using translates of the edges in the Newton polygons for $g(x)$ and $h(x)$ with respect to the prime $p$, using exactly one translate for each edge of the Newton polygons for $g(x)$ and $h(x).$ Neccessarily, the translated edges are translated in such a way as to form a polygonal path with the slopes of the edges increasing.
\end{theorem}

It is, therefore, natural to ask the following questions: \textit{Given a polynomial \( g \), how does the Newton polygon of \( g \circ f \) change as \( f \) varies?  Can we describe families of polynomials that, when composed with \( g \), preserve key features (such as segments and slopes) of the Newton polygon?} 




In this paper, we attempt to answer the above questions. Our main result provides a clear method to construct the Newton polygon of \( g \circ f \) from the structure of the Newton polygons of \( g \) and \( f \), where both \( g \) and \( f \) can have Newton polygons with multiple segments. Specifically, this construction works under the condition that the slopes of the Newton polygon of \( g \) are not excessively steep, or the \( p \)-adic valuation of the constant term of \( g \) is bounded relative to the slope of its first edge. 
Composition with a polynomial \( f \) has the effect of ``stretching'' the Newton polygon of \( g \) horizontally by a factor of \( \deg f \). This transformation preserves the number of segments in the Newton polygon and adjusts the slopes in a predictable manner. 

We now state our first main result which is proved in Section \ref{sec3}.

\begin{theorem}\label{main thm1}
Let \( p \) be a prime number and let \( g(x) = b_e x^e + b_{e-1} x^{e-1} + \cdots + b_0 \), where \( b_0 \neq 0 \), be a polynomial of degree \( e \) with rational coefficients such that \( \nu_p(b_e)=0 \) and \( p \) divides \( b_i \) for \( 0 \leq i \leq e-1 \). Let \( m_0, m_1, \dots, m_{t-1}, m_t \) be integers such that the successive vertices of the Newton polygon of \( g(x) \) with respect to \( p \) are given by the set 
\[
\{(0,0), (e-m_1, \nu_p(b_{m_1})), \dots, (e-m_{t-1}, \nu_p(b_{m_{t-1}})), (e, \nu_p(b_0))\},
\]
with \( m_0 = e \), \( m_t = 0 \), and  \( \lambda_i = \frac{\nu_p(b_{m_i}) - \nu_p(b_{m_{i-1}})}{m_{i-1} - m_i} \) for \( 1 \leq i \leq t \). 

Assume that \( f(x) \in \mathbb{Q}[x] \) is a polynomial of degree \( d \) such that the slope \( \lambda \) of the first edge of the Newton polygon of \( f(x) \) with respect to \( p \) is greater than or equal to \( \lambda_1 \) and that the \( p \)-adic valuation of the leading coefficient of \( f(x) \) is zero. If one of the following conditions holds:
\begin{itemize}
    \item[(i)] \( \nu_p(b_0) < \lambda_1(d + e - 1), \)
    \item[(ii)] \( \nu_p(b_0) = \lambda_1(d + e - 1) \) with \( \nu_p(f(0)) > d\lambda \) or \( \lambda > \lambda_1 \),
\end{itemize}
then for any positive integer \( n \), the successive vertices of the Newton polygon of the composition \( g(f^n(x)) \) with respect to \( p \) are given by the set
\[ \{(0,0), (d^n(e - m_1), \nu_p(b_{m_1})), \dots, (d^n(e - m_{t-1}), \nu_p(b_{m_{t-1}})), (d^n e, \nu_p(b_0))\}, \]
where \( f^n(x) \) denotes the \( n \)-th iterate of the polynomial \( f(x) \).
\end{theorem}

For a polynomial $f(x) \in \Q[x]$, the following result predicts the Newton polygon structure of $f^n$ based on the Newton polygon structure of $f$.

\begin{theorem}
	Let $p$ be a prime and $f(x) = a_d x^d + \cdots + a_1x + a_0 \in \Q[x]$ be a polynomial of degree $d$, with $a_0 \neq 0$ and $\nu_p(a_{m_1})>0$. Suppose that the successive vertices of $NP_p(f)$ are given by the set $\{(0,0), (d-m_1, v_p(a_{m_1})), \dots, (d-m_{t-1}, v_p(a_{m_{t-1}})), (d, v_p(a_0))\}$. If $\nu_p(a_0) \leq \frac{v_p(a_{m_1})}{d-m_1} (2d-1)$, then for $n \geq 1$, the successive vertices of $NP_p(f^n)$ are given by the set $\{(0,0), (d^{n-1}(d-m_1), v_p(a_{m_1})), \dots, (d^{n-1}(d-m_{t-1}), v_p(a_{m_{t-1}})), (d^{n}, v_p(a_0))\}$. In particular, $f(x)$ is eventually stable.
\end{theorem}

\noindent \textbf{Remark.} We would like to point out that the condition \( \nu_p(b_e) = 0 \) is redundant. Suppose \( \nu_p(b_e) \neq 0 \), and replace assumption (i) of Theorem \ref{main thm1} with the condition \( \nu_p(b_0) - \nu_p(b_e) < \lambda_1 (d + e - 1) \). Define the polynomial \( g_1(x) = p^{-\nu_p(b_e)} g(x) \). It is straightforward to verify that \( \nu_p(g_1(0)) = \nu_p(b_0) - \nu_p(b_e) < \lambda_1 (d + e - 1) \), which satisfies the conditions of Theorem \ref{main thm1}. We can then apply Theorem \ref{main thm1} to \( g_1(x) \) and determine the Newton polygon structure of the composition \( (g_1 \circ f^n)(x) \). This, in turn, provides the Newton polygon of \( g(f^n(x)) = p^{\nu_p(b_e)} g_1(f^n(x)) \) for all \( n \geq 0 \).\\

A polynomial \( f(x) = a_d x^d + a_{d-1} x^{d-1} + \cdots + a_0 \in \Q[x] \) is said to be \( p^r\textit{-pure} \) for a prime \( p \) and an integer \( r \geq 1 \) if \( \nu_p(a_d) = 0 \), \( \nu_p(a_0) = r \), and \( \frac{\nu_p(a_i)}{d-i} \geq \frac{r}{d} \) for all \( 0 < i < d \). \\

Note that in Theorem \ref{main thm1}, if $g(x)$ has a single edge, condition (i) is always satisfied. Consequently, the main result of \cite{darwish sadek} can be derived as a corollary of Theorem \ref{main thm1} combined with Dumas' Theorem.


\begin{corollary} \label{cor1.2}
    Suppose \( f(x) \in \Q[x] \) is a \( p^r \)-pure polynomial of degree \( d \). Then for any \( n \geq 1 \), the iterate \( f^n(x) \) has at most \( \gcd(d^n, r) \) irreducible factors over \( \Q \), and each irreducible factor has degree at least \( \frac{d^n}{\gcd(d^n, r)} \).
\end{corollary}

A \( p^r \)-pure polynomial \( f \) of degree \( d \) is said to be \textit{\( p^r \)-Dumas} if \( \gcd(r, d) = 1 \). The following corollary regarding \( p^r \)-Dumas polynomials is an immediate consequence of Corollary \ref{cor1.2}.

\begin{corollary} \label{cor6.2}
    Suppose \( f(x) \in \Q[x] \) is a \( p^r \)-Dumas polynomial of degree \( d \). Then \( f^n(x) \) is irreducible over \( \Q \) for all \( n \geq 1 \).
\end{corollary}

Note that the application of Theorem \ref{main thm1} relies on the condition that all edges of the Newton polygon of \( g \) have positive slopes with respect to a prime \( p \). Our next result extends this stretching property of the Newton polygon to cases where the edges can have negative slopes, provided certain additional restrictions are imposed on the polynomial \( f \).


\begin{theorem}\label{main thm2}
Let \( p \) be a prime number, and let \( g(x) = b_e x^e + b_{e-1} x^{e-1} + \cdots + b_0 \) with \( b_0 \neq 0 \) be a polynomial of degree \( e \) with rational coefficients. Let \( m_0, m_1, \dots, m_{t-1}, m_t \) be integers such that the successive vertices of the Newton polygon of \( g(x) \) with respect to \( p \) are given by the set 
\[
\{(0, \nu_p(b_e)), (e - m_1, \nu_p(b_{m_1})), \dots, (e - m_{t-1}, \nu_p(b_{m_{t-1}})), (e, \nu_p(b_0))\},
\]
with \( m_0 = e \), \( m_t = 0 \), and \( \lambda_i = \frac{\nu_p(b_{m_i}) - \nu_p(b_{m_{i-1}})}{m_{i-1} - m_i} \) for \( 1 \leq i \leq t \). 

Let \( u \) be an integer such that \( u > |\lambda_j| \) for all \( j \in \{1,2, \ldots, t\} \). Suppose \( f(x) = a_d x^d + \cdots + a_1 x + a_0 \) is a polynomial of degree \( d \) such that 
\( \nu_p(a_d)=0 \) and $\nu_p(a_i) \geq \frac{u}{\beta} (d-i)$ for all $i \in \{ 0,1,2, \ldots, d \}$, 
where $\beta \leq d$ is a positive integer coprime to $u$. Then, the successive vertices of the Newton polygon of the composition \( g(f(x)) \) with respect to \( p \) are given by the set
\[
\{(0, \nu_p(b_e)), (d(e - m_1), \nu_p(b_{m_1})), \dots, (d(e - m_{t-1}), \nu_p(b_{m_{t-1}})), (de, \nu_p(b_0))\}.
\]
\end{theorem}

\noindent The proof of above theorem is given in Section \ref{sec4}. The following corollary follows as a direct consequence of Theorem \ref{main thm2}.

\begin{corollary}
	Under the notations and assumptions of Theorem \ref{main thm2}, the successive vertices of the Newton polygon of \( (g \circ f^n)(x) \) with respect to \( p \) are given by the set
	\[ 
	\{(0, \nu_p(b_e)), (d^n(e - m_1), \nu_p(b_{m_1})), \dots, (d^n(e - m_{t-1}), \nu_p(b_{m_{t-1}})), (d^ne, \nu_p(b_0))\}. 
	\]
\end{corollary}

In 2024, Gajek-Leonard and Tomer \cite{gajek tomer} proved an interesting result regarding the Newton polygon of composition of polynomials, which can be deduced as an easy corollary of Theorem \ref{main thm2}.

\begin{corollary}\cite[Theorem 1.1]{gajek tomer}
	Let \( f(x), g(x) \in \Q[x] \) be polynomials. Suppose the Newton polygon of \( g(x) \) with respect to \( p \) consists of \( t \) segments with slopes \( \lambda_1 < \cdots < \lambda_t \). If \( f(x) \) is \( p^r \)-pure and \( |\lambda_i| < r \) for all \( 1 \leq i \leq t \), then the Newton polygon of \( (g \circ f)(x) \) has \( t \) segments with slopes \( \frac{\lambda_1}{\deg f} < \cdots < \frac{\lambda_t}{\deg f} \).
\end{corollary}

In 1985, Odoni \cite{Odoni85} studied a recurrence sequence defined by relations of the type \( a_n = f(a_{n-1}) \), where \( a_0 \in \mathbb{Z} \) and \( f(x) \in \mathbb{Z}[x] \) is a polynomial of degree at least two. Such a sequence can be described as the orbit of \( a_0 \) under the dynamical system generated by iterations of \( f(x) \). A natural question is whether the sequence \( (a_n)_{n \geq 0} \) contains infinitely many primes, and this remains an open problem. The sequence \( (a_n) \) grows extremely rapidly, approximately to the order of \( d^{d^n} \), and heuristics suggest that only finitely many terms are prime. For instance, taking \( a_0 = 3 \) and \( f(x) = (x-1)^2 + 1 \) yields the Fermat numbers, whose primality have been a mystery since Fermat first speculated about them in 1640. 

A more reasonable hope is to obtain some qualitative information about the prime factorizations of $a_n$, for instance by considering the whole collection
\[ P(g,f,a_0) := \{\text{prime divisors of non-zero terms of } (g(f^n(a_0)))_{n \geq 0} \}. \]

If this set is sparse within the set of all primes, then at least the sequence $(g(f^n(a_0)))_{n \geq 0}$ do not in the aggregate have too many small prime factors. In 2008, Jones \cite[Theorem 1.1]{jones08} demonstrated that for monic polynomials \( f(x), g(x) \in \Z[x] \), with \( f \) quadratic and under certain conditions, the natural density of \( P(g, f, a_0) \) is zero. Similar results were proved by Hamblen \textit{et al.} \cite{HJM2014} in 2014 for binomials of the type $z^d+c$. One of the crucial property required in both the papers were the stability or eventual stability of the sequences which can be easily deduced from Theorems \ref{main thm1} and \ref{main thm2}.  For example, Theorem 1.6 of \cite{HJM2014} (with \( K = \Q \)) follows as a corollary of Theorem \ref{main thm1}. Beyond stability, these theorems provide detailed insights into the number and degrees of irreducible factors of iterates, as elaborated in Section \ref{sec6.3}.

The eventual stability of polynomials is a central theme in the arithmetic of dynamical systems and plays a key role in proving Sookdeo's conjecture \cite{sook11}, which we discuss in Section \ref{sec6.9}. Eventual stability has also been pivotal in recent proofs of finite-index results for arboreal representations (see \cite{tucker19, tucker21}). It also has a lot of applications in preimage curves and arboreal Galois representations, as detailed in \cite[Section 3]{jones levy}.

A polynomial \( f(x) \in \Q[x] \) of degree \( d \) is said to be \textit{stable} (or \textit{dynamically irreducible}) over \( \Q \) if \( f^n(x) \) is irreducible over \( \Q \) for all \( n \geq 1 \). It is said to be \textit{eventually stable} if there exists a constant \( C_f \), depending only on \( f \), such that the number of irreducible factors of \( f^n(x) \) is bounded by \( C_f \). From Corollary \ref{cor1.2}, it follows that a \( p^r \)-pure polynomial is always eventually stable with \( C_f = \max_{n \in \N} \gcd(d^n, r) < r \). In fact, taking $f$ and $g$ to be same in Theorem \ref{main thm1}, we can see that, $f$ is eventually stable when $f$ satisfies the conditions of Theorem \ref{main thm1}.

Our next result provides eventual stability to a broader class of polynomials compared to Theorems \ref{main thm1} and \ref{main thm2}, although it does not provide explicit details about the Newton polygon structure or the number and degrees of irreducible factors of \( f^n(x) \) that Theorems \ref{main thm1}, \ref{main thm2} offer.



\begin{theorem} \label{p5.2}
	Let $p$ be a prime, and let $f(x) = a_d x^d + a_{d-1} x^{d-1} + \cdots + a_1 x + a_0,~ a_0 \neq 0$ be a polynomial of degree $d$ with rational coefficients such that $\nu_p(a_i) > 0$ for all $i$, $0 \leq i < d$, and $\nu_p(a_d) = 0$. Then $f^n(x)$ is eventually stable over $\mathbb{Q}$.
\end{theorem}





The proof of Theorem \ref{p5.2} is given in Section \ref{sec5}. 

For a polynomial $g(x) \in \mathbb{Q}[x]$, we say $f$ is \textit{dynamically irreducible at} $g$ if $g \circ f^n$ is irreducible for all $n \geq 0$.

Let $m$ be a positive integer, and let $b_0, b_1, \ldots, b_m$ denote arbitrary integers such that $|b_0| = |b_m| = 1$. A \textit{Schur polynomial} of degree $m$, denoted by $G_m$, is defined as:
\[
    G_m(x) = b_0 + b_1 x + b_2 \frac{x^2}{2!} + \cdots + b_m \frac{x^m}{m!}.
\]

It is well known that $G_m$ is irreducible for all $m$ (see \cite{schur29 1}, \cite[Theorem 2]{filaseta bessel}).

In the special case where $b_i = 1$ for all $i, 0 \leq i \leq m$, $G_m$ becomes the \textit{truncated exponential polynomial} $E_m$. An independent proof of the irreducibility of truncated exponential polynomials using Newton polygons was given by Coleman in 1987 (cf. \cite{coleman}).

In Section \ref{sec6.1}, we prove the following result, which demonstrates how Theorem \ref{main thm2} can be utilized to construct a large family of polynomials $f$ that are dynamically irreducible at \textit{Schur polynomials}, and consequently, at \textit{truncated exponential polynomials}, irrespective of whether these polynomials are irreducible themselves.

\begin{theorem} \label{prop1.6}
    Fix $m \geq 1$ and let $G_m(x)$ be the Schur polynomial of degree $m$. Let $f(x) \in \mathbb{Q}[x]$ be a polynomial of degree $d$ with leading coefficient $a_d$ such that $\nu_p(a_d)=0$ and $f(x) \equiv a_d x^d \pmod{p}$. If $d$ is coprime to $m!$ and $\gcd(b_i, m) = 1$ for all $0 < i < m$, then $f$ is dynamically irreducible at $G_m$.
\end{theorem}

The following result follows immediately from Theorem \ref{prop1.6}.

\begin{corollary}\label{prop1.5}
    Let $f$ and $d$ be as in Theorem \ref{prop1.6}, and let $E_m$ denote the truncated exponential polynomial of degree $m$. If $d$ is coprime to $m!$, then $f$ is dynamically irreducible at $E_m$.
\end{corollary}

In Section \ref{sec6.2}, we use Theorem \ref{main thm1} to construct an iterative family of non-monogenic polynomials. The Section \ref{sec6.3} applies Theorems \ref{main thm1} and \ref{main thm2} to analyze the number of irreducible factors and their degrees in the iterates of certain polynomials. Finally, in Section \ref{sec6.9}, we present a family of examples that satisfy Sookdeo's Conjecture, showcasing an application of Theorems \ref{main thm1} and \ref{main thm2}.



\section{Preliminaries}
Throughout this paper, $p$ will always denote a prime number and we use the notation \( NP_p(h) \) to denote the Newton polygon of \( h(x) \in \mathbb{Q}[x] \) with respect to \( p \). To prove Theorems \ref{main thm1} and \ref{main thm2}, we will need the following four lemmas.

\begin{lemma} \label{lemma2.1}
Let \( g(x) = b_e x^e + b_{e-1} x^{e-1} + \cdots + b_0 \), with \( b_0 \neq 0 \), be a polynomial of degree \( e \) with rational coefficients such that \( \nu_p(b_e)=0 \) and \( p \) divides \( b_s \) for \( 0 \leq s \leq e-1 \). Let \( m_0, m_1, \dots, m_{t-1}, m_t \) be integers such that the successive vertices of the Newton polygon of \( g(x) \) with respect to \( p \) are given by the set
\[
\{(0,0), (e - m_1, \nu_p(b_{m_1})), \dots, (e - m_{t-1}, \nu_p(b_{m_{t-1}})), (e, \nu_p(b_0))\},
\]
where \( m_0 = e \) and \( m_t = 0 \). If \( \alpha \) is a positive integer such that \( \nu_p(b_0) \leq \frac{\nu_p(b_{m_1})}{e - m_1} \alpha \), then for any \( s \) with \( 0 \leq s \leq t \), we have
\begin{align} \label{eqn:2.1}
    \nu_p(b_{m_s}) \leq \frac{\nu_p(b_{m_1})}{e - m_1} (\alpha - m_s).
\end{align}
\end{lemma}

\begin{proof}
Denote \( \nu_p(b_{m_s}) \) by \( r_s \) for \( 0 \leq s \leq t \). Using the fact that \( m_t = 0 \) and the hypothesis \( \nu_p(b_0) \leq \frac{r_1}{e - m_1} \alpha \), it is clear that \eqref{eqn:2.1} holds for \( s = t \). 

Now, let us fix an integer \( s \) with \( 0 \leq s < t \). As the slopes of the edges of the Newton polygon of \( g(x) \) with respect to \( p \) are in increasing order, we can easily verify that
\[
\frac{r_1}{e - m_1} \leq \frac{r_{s+1}-r_s}{m_s-m_{s+1}} \leq \frac{r_t - r_s}{e-(e-m_s)}.
\]
Using the above inequality along with the hypothesis \( r_t \leq \frac{r_1}{e - m_1} \alpha \), we obtain
\[
r_s \leq r_t - \frac{r_1}{e - m_1} m_s \leq \frac{r_1}{e - m_1} \alpha - \frac{r_1}{e - m_1} m_s.
\]
This completes the proof of the lemma.
\end{proof}

\begin{lemma}\label{lemma2.2}
Let \( g(x) = b_e x^e + b_{e-1} x^{e-1} + \cdots + b_0  \in \Q[x]\), with \( b_0 \neq 0 \), be a polynomial of degree \( e \). Let \( m_0, m_1, \dots, m_{t-1}, m_t \) be integers such that the successive vertices of the Newton polygon of \( g(x) \) with respect to \( p \) are given by the set
\[
\{(0, \nu_p(b_e)), (e - m_1, \nu_p(b_{m_1})), \dots, (e - m_{t-1}, \nu_p(b_{m_{t-1}})), (e, \nu_p(b_0))\},
\]
where \( m_0 = e \), \( m_t = 0 \), and the slopes of the segments of the Newton polygon of $g$ are given by
\[
\lambda_i = \frac{\nu_p(b_{m_i}) - \nu_p(b_{m_{i-1}})}{m_{i-1} - m_i} \quad \text{for } 1 \leq i \leq t.
\]
Let \( s \) and \( \alpha \) be non-negative integers. Then we have the following:
\begin{itemize}
    \item[(i)] If \( s + 1 < \alpha \leq t \), then
    \[
    \nu_p(b_{m_{s+1}}) + \lambda_{s+1} (m_{s+1} - j) < \nu_p(b_{m_\alpha}) + \lambda_{\alpha} (m_{\alpha} - j), \quad \text{for all } j \leq m_{\alpha}.
    \]
    \item[(ii)] If \( 0 \leq \alpha \leq s + 1 \leq t \), then
    \[
    \nu_p(b_{m_{s+1}}) = \nu_p(b_{m_\alpha}) + \sum_{i=\alpha+1}^{s+1} \lambda_i (m_{i-1} - m_i).
    \]
\end{itemize}
\end{lemma}

\begin{proof}
Denote \( \nu_p(b_{m_i}) \) by \( r_i \) for \( 0 \leq i \leq t \). We will prove the two cases separately. \\

\noindent \textbf{Case (i):} Suppose \( s + 1 < \alpha \). Since the slopes of the edges of the Newton polygon of \( g(x) \) with respect to \( p \) are in increasing order, we have
\[
\lambda_{s+1} < \frac{r_{\alpha} - r_{s+1}}{m_{s+1} - m_{\alpha}} \quad \text{and} \quad \lambda_{s+1} < \lambda_{\alpha}.
\]
Using these inequalities and the identity \( \lambda_{s+1} (m_{s+1} - j) = \lambda_{s+1} (m_{s+1} - m_{\alpha}) + \lambda_{s+1} (m_{\alpha} - j) \) for \( j \leq m_{\alpha} \), it follows that
\[
\lambda_{s+1} (m_{s+1} - j) < r_{\alpha} - r_{s+1} + \lambda_{\alpha} (m_{\alpha} - j),
\]
which can be rearranged as
\[
r_{s+1} + \lambda_{s+1} (m_{s+1} - j) < r_{\alpha} + \lambda_{\alpha} (m_{\alpha} - j).
\]
This completes the proof of Case (i).

\noindent \textbf{Case (ii):} Suppose \( 0 \leq \alpha \leq s + 1 \leq t \). Using the definition of $\lambda_i$, we have $r_{s+1} = r_s + \lambda_{s+1} (m_s - m_{s+1})$. With iterative similar equalities, we obtain 
\[
\nu_p(b_{m_{s+1}}) = \nu_p(b_{m_\alpha}) + \sum_{i=\alpha+1}^{s+1} \lambda_i (m_{i-1} - m_i).
\]
This completes the proof of the lemma.
\end{proof}

\begin{lemma}\label{lemma2.3}
With the notations and assumptions of Theorem \ref{main thm1}, denote \( \nu_p(b_{m_i}) \) by \( r_i \) for \( 0 \leq i \leq t \). Assume that the successive vertices of the Newton polygon of \( (g \circ f^n)(x) \) with respect to \( p \) are given by the set
\[
\{(0,0), (d^n(e-m_1), r_1), \dots, (d^n(e-m_{t-1}), r_{t-1}), (d^n e, r_t)\}.
\]
If the polynomial \( (g \circ f^{n+1})(x) \) is given by \( (g \circ f^{n+1})(x) = \sum\limits_{k=0}^{d^{n+1}e} C_k x^k \), then for \( k = d^{n+1} m_s \) with \( 0 \leq s \leq t \), we have \( \nu_p(C_k) = r_s \).
\end{lemma}

\begin{proof}
Let \( f(x) = \sum_{i=0}^{d} A_i x^i \) and \( (g \circ f^n)(x) = \sum_{j=0}^{d^n e} B_j x^j \). Composition of \( (g \circ f^n)(x) \) and $f$ is given by
\[
(g \circ f^{n+1})(x) = \sum_{j=0}^{d^n e} B_j \left( \sum_{i=0}^{d} A_i x^i \right)^j = \sum_{k=0}^{d^{n+1}e} C_k x^k,
\]
For given \( k = d^{n+1}m_s \), we determine \( \nu_p(C_k) \) by analyzing the \( p \)-adic valuation of the coefficient of \( x^k \) in the expansion of the term:
\begin{equation} \label{e2.2}
B_j \left( \sum_{i=0}^{d} A_i x^i \right)^j
\end{equation}
for each \( j \) in the range \( 0 \leq j \leq d^n e \).

By hypothesis, the successive vertices of the Newton polygon of \( (g \circ f^n)(x) \) are given by the set
\[
\{(0,0), (d^n(e-m_1), r_1), \dots, (d^n(e-m_{t-1}), r_{t-1}), (d^n e, r_t)\}.
\]
This means that \( NP_p(g \circ f^n) \) will appear as shown in Figure \ref{NP of gofn} below.

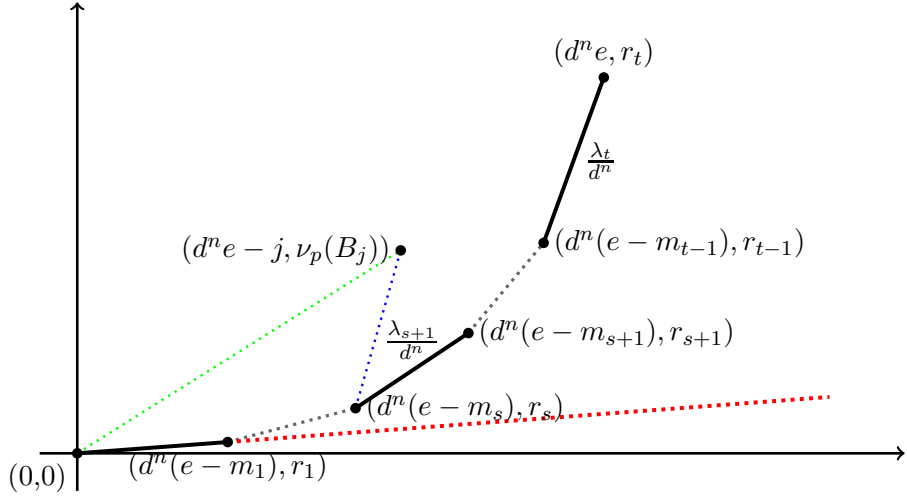
\begin{figure}[h]
    \centering
    \begin{tikzpicture}
        \draw[->, line width=1pt] (0,-0.5) -- (0,6);
        \draw[->, line width=1pt] (-0.5,0) -- (11,0);
        
        \coordinate (A) at (0,0);
        \coordinate (B) at (2,.15);
        \coordinate (C) at (3.7,0.6);
        \coordinate (D) at (5.2,1.6);
        \coordinate (E) at (6.2,2.8);
        \coordinate (F) at (7,5);
        
        \coordinate (I) at (4.3,2.7);
        
        \draw[line width=1.5pt] (A) -- (B);
        \draw[dotted, line width=1.2pt, black!60] (B) -- (C);
        \draw[line width=1.5pt] (C) -- (D) node[midway, above] {$\frac{\lambda_{s+1}}{d^n}$};
        \draw[dotted, line width=1.2pt, black!60] (D) -- (E);
        \draw[line width=1.5pt] (E) -- (F) node[midway, right] {$\frac{\lambda_t}{d^n}$};
        
        \draw[dotted, line width=1.5pt, red] (B) -- ++(8,0.6);
        \draw[dotted, line width=0.9pt, green] (A) -- (I);
        \draw[dotted, line width=0.9pt, blue] (C) -- (I);
        
        \foreach \point/\position/\name in {A/below left/{(0,0)}, B/below/{$(d^n(e-m_1), r_1)$}, C/right/{$(d^n(e-m_{s}), r_{s})$}, D/right/{$(d^n(e-m_{s+1}), r_{s+1})$}, E/right/{$(d^n(e-m_{t-1}), r_{t-1})$}, F/above/{$(d^ne, r_t)$}, I/left/{$(d^ne-j,\nu_p(B_j))$}} {
            \fill (\point) circle (2pt) node [\position] {\name};
        }
    \end{tikzpicture}
    \caption{Newton polygon of \( g \circ f^n(x) \) with respect to $p$}
    \label{NP of gofn}
\end{figure}

Recall that the Newton polygon of \( g \circ f^n \) is the lower convex hull formed using the points \( \{(d^n e - j, \nu_p(B_j)) : 0 \leq j \leq d^n e\} \). Therefore, for any \( 0 < j \leq d^n e \), the point \( (d^n e - j, \nu_p(B_j)) \) will lie on or above the Newton polygon, i.e.,
\[
\frac{\nu_p(B_j) - 0}{d^n e - j - 0} \geq \frac{r_1}{d^n(e - m_1)} = \frac{\lambda_1}{d^n}, \quad \text{for all } j \in \{0, 1, 2, \dots, d^n e - 1\}.
\]
This implies that
\begin{equation}
\nu_p(B_j) \geq \frac{\lambda_1}{d^n} (d^n e - j), \quad \text{for all } j \in \{0, 1, 2, \dots, d^n e\}. \label{e2.3}
\end{equation}
	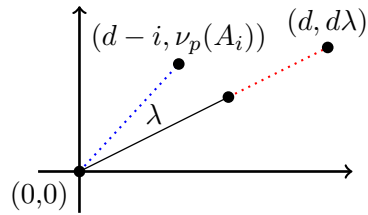
\begin{figure}[!htb]
    \centering
    \begin{minipage}{0.39\textwidth}
        From the hypothesis, the first edge of $NP_p(f)$ has slope $\lambda$. Thus, for $0 < i \leq d$, the point $(d-i, \nu_p(A_i))$ lies above the line joining $(0,0)$ and $(d, d\lambda)$, as shown in Figure \ref{NP of f}.     \end{minipage}
    \hfill
    \begin{minipage}{0.6\textwidth}
        \centering
        \begin{tikzpicture}[scale=1.1]
            \draw[->, line width=1pt] (0,-0.5) -- (0,2); 
            \draw[->, line width=1pt] (-0.5,0) -- (3.3,0); 
            
            \coordinate (A) at (0,0);
            \coordinate (B) at (1.8,0.9);
            \coordinate (C) at (3,1.5);
            \coordinate (D) at (1.2,1.3);
            
            \draw[line width=.5pt] (A) -- (B) node[midway, above] {$\lambda$};
            \draw[dotted, line width=.8pt, red] (B) -- (C);
            
            \draw[dotted, line width=.8pt, blue] (A) -- (D);
            
            \foreach \point/\position/\name in {A/below left/{(0,0)}, B/below/{ }, C/above/{$(d,d\lambda )$}, D/above/{$(d-i, \nu_p(A_i))$}} {
                \fill (\point) circle (2pt) node [\position] {\name}; }
        \end{tikzpicture}
        \caption{Newton polygon of $f$ with respect to $p$}
        \label{NP of f}
    \end{minipage}
\end{figure}
\vspace{1.2cm}

\noindent This implies
\begin{align*}
    \frac{\nu_p(A_i) - 0}{d - i - 0} &\geq \lambda, \quad \text{for all } i \in \{0, 1, 2, \ldots, d-1\};
\end{align*}
which yields
\begin{equation}
    \nu_p(A_i) \geq \lambda (d - i), \quad \text{for all } i \in \{0, 1, 2, \ldots, d-1, d\}. \label{e2.4}
\end{equation}

The proof of the lemma is divided into three cases. We first prove the lemma for the cases $s=0$ and $s=t$ seperately, then proceed with the case $0 < s < t$.

\noindent \textbf{Case $s=0$:} We begin with $s=0$, i.e., $k = d^{n+1}m_0 = d^{n+1}e$. In this situation, the only term of degree $k = d^{n+1}e$ in the expression of $(g \circ f^{n+1})(x)$ is $B_{d^ne}A_d^{d^ne}x^{d^{n+1}e}$, thus we have
\[
\nu_p(C_k) = \nu_p(B_{d^ne}) + d^ne  \nu_p(A_d) = 0 = r_0.
\]
This proves the lemma for $s=0$.

\noindent \textbf{Case $s = t$:} Now consider $s = t$, i.e., $k = d^{n+1}m_t = 0$. Then we have $C_k = C_0 = \sum_{j=0}^{d^n e} B_j A_0^j$. Recall the given fact that $r_t \leq \lambda_{1}(d + e - 1)$. For any $j \leq d^n e$, we use the upper bound of $r_t$ and note that $-d^n \leq -1$ to obtain
\[
r_t \leq \frac{\lambda_{1}}{d^n}(d^{n+1} + d^ne - d^n) \leq \frac{\lambda_1}{d^n} (d^{n+1} + d^n e - 1) = \frac{\lambda_1}{d^n} (d^n e - j) + \frac{\lambda_1}{d^n} (d^{n+1} + j - 1).
\]
Note that $d^{n+1} \geq 1$, and it is easy to verify that for $j \geq 1$, $d^{n+1} + j - 1 \leq d^{n+1} j$. By applying this, along with $\lambda_1 \leq \lambda$ and using Equations \eqref{e2.3} and \eqref{e2.4}, we obtain:
\[
r_t \leq \lambda_{1}(d + e - 1) \leq \frac{\lambda_1}{d^n} (d^n e - j) + \frac{\lambda}{d^n} d^{n+1} j \leq \nu_p(B_j) + j \nu_p(A_0) = \nu_p(B_j A_0^j).
\]
Remember that the above inequality holds only for $j \geq 1$. Furthermore, we note that one of the inequalities above is strict, based on the assumptions (i) and (ii) of Theorem \ref{main thm1}; hence, we have $\nu_p(B_j A_0^j) > r_t$ whenever $j \geq 1$. Also, for $j = 0$, we have $\nu_p(B_0 A_0^0) = \nu_p(B_0) = r_t$. Hence, we conclude that $\nu_p(C_0) = r_t$, which proves the lemma for $s = t$ as well.

\noindent \textbf{Case $0 < s < t$:} Fix an integer $s$ satisfying $0 < s < t$, and consider $k = d^{n+1}m_s$. If $j < \frac{k}{d}$, then the highest possible degree of $x$ in the expansion of \eqref{e2.2} is $dj$, which is strictly smaller than $k$. Therefore, no term of degree $k$ exists in \eqref{e2.2} whenever $j < \frac{k}{d}$.

Now, if $j > \frac{k}{d} = d^n m_s$, we observe that $j - \frac{k}{d} = j - d^n m_s \geq 1$. Suppose that $k$ can be partitioned into $j$ terms as $k = i_1 + i_2 + \cdots + i_j$ with each $i_{\ell} \leq d$ for $1 \leq \ell \leq j$. Then there exists a term in the form $B_j \prod_{\ell=1}^{j} A_{i_{\ell}} x^{i_{\ell}} = \left( B_j \prod_{\ell=1}^{j} A_{i_{\ell}} \right) x^k$ of degree $k$ in the expansion of \eqref{e2.2}. Now, we show that for each such partition of $k$ with $j > \frac{k}{d}$, the inequality $\nu_p(C_k) > r_{s}$ holds. Using Equations \eqref{e2.3} and \eqref{e2.4}, we obtain
\[
\nu_p(C_k) \geq \nu_p\left( B_j \prod_{\ell=1}^{j} A_{i_{\ell}} \right) = \nu_p(B_j) + \sum_{\ell=1}^{j} \nu_p(A_{i_{\ell}}) \geq \frac{\lambda_1}{d^{n}} (d^n e - j) + \sum_{\ell=1}^{j} \lambda (d - i_{\ell}).
\]
Keeping in mind the given fact $\lambda \geq \lambda_1$ and the condition $d - i_{\ell} \geq 0$, we deduce
\begin{align}
    \nu_p(C_k) &\geq \frac{\lambda_1}{d^{n}} (d^n e - j) + \lambda_1 \sum_{\ell=1}^{j} (d - i_{\ell}) \nonumber \\
    &= \lambda_1 \left[ e - \frac{j}{d^n} + dj - \sum_{{\ell}=1}^{j} i_{\ell} \right] \nonumber \\
    &= \lambda_1 \left[ e - \frac{j}{d^n} + dj - k + \frac{k}{d^{n+1}} - m_{s+1} \right] - \lambda_1 \left( \frac{k}{d^{n+1}} - m_{s+1} \right) \label{e2.5} \\
    &> \lambda_1 \left[ e + \left( d - \frac{1}{d^n} \right) \left( j - \frac{k}{d} \right) - m_{s+1} \right] - \lambda_{s+1} \left( \frac{k}{d^{n+1}} - m_{s+1} \right) \nonumber
\end{align}
where the last inequality is strict because for any $s > 0$, we have $\lambda_1 < \lambda_{s+1}$ and $\frac{k}{d^{n+1}} - m_{s+1} = m_{s} - m_{s+1} > 0$. Additionally, using $j - \frac{k}{d} \geq 1$ and $d \geq 1$, we derive
\[
\nu_p(C_k) > \lambda_1 \left[ e + (d - 1) - m_{s+1} \right] - \lambda_{s+1} \left( \frac{k}{d^{n+1}} - m_{s+1} \right).
\]

\noindent Recall the given fact that $r_t \leq \lambda_1 (d + e - 1)$. Applying Lemma \ref{lemma2.1} with $\alpha = e + d - 1$ for the first part of RHS and substituting the values of $k$, $\lambda_{s+1}$, we obtain
\[
\nu_p(C_k) > r_{s+1} - \left( \frac{r_{s+1} - r_s}{m_s - m_{s+1}} \right) \left(\frac{d^{n+1} m_s}{d^{n+1}} - m_{s+1} \right) = r_s.
\]
Thus, for $j > \frac{k}{d} = d^n m_s$, we assert that $\nu_p(C_k) > r_s$. Now for the case $j = \frac{k}{d}$, i.e., $j = d^n m_s$, the only term of degree $k$ in the expansion of \eqref{e2.2} is $A_d^{d^n m_s} B_{d^n m_s}$. Furthermore, by applying the hypothesis to the structure of $NP_p(g \circ f^n)$ and recalling that $p \nmid A_d$, we derive $\nu_p(A_d^{d^n m_s} B_{d^n m_s}) = d^n m_s \nu_p(A_d) + \nu_p(B_{d^n m_s}) = d^n m_s (0) + r_s = r_s$. Combining all of these, we obtain $\nu_p(C_{d^{n+1} m_s}) = r_s$. As $s$ was chosen arbitrarily, the lemma holds for all $0 < s < t$. This completes the proof of the lemma.
\end{proof}

\begin{lemma}\label{lemma2.4}
With the notations and hypotheses of Theorem \ref{main thm1}, denote $\nu_p(b_{m_i})$ by $r_i$ for $0 \leq i \leq t$. Assume that the successive vertices of the Newton polygon of $(g \circ f^n)(x)$ with respect to $p$ are given by the set
\[
\{(0,0), (d^n(e-m_1), r_1), \dots, (d^n(e-m_{t-1}), r_{t-1}), (d^n e, r_t)\}.
\]
If the polynomial $(g\circ f^{n+1})(x)$ is given by $(g\circ f^{n+1})(x) = \sum\limits_{k=0}^{d^{n+1}e} C_k x^k$, then for $k \leq d^{n+1}m_s$ with $0 \leq s < t$, we have
\[
\nu_p(C_k) \geq r_{s+1} + \lambda_{s+1} \left( m_{s+1} - \frac{k}{d^{n+1}} \right).
\]
\end{lemma}

\begin{proof}
Write $f(x) = \sum_{i=0}^{d} A_i x^i$ and $(g \circ f^n)(x) = \sum_{j=0}^{d^n e} B_j x^j$. Since $(g\circ f^{n+1})(x)$ is given by $(g\circ f^{n+1})(x) = \sum\limits_{k=0}^{d^{n+1}e}C_kx^k$, we obtain
\[
(g\circ f^{n+1})(x) = \sum_{j=0}^{d^n e} B_j \left(\sum_{i=0}^{d} A_i x^i\right)^j =  \sum\limits_{k=0}^{d^{n+1}e}C_kx^k.
\]
To determine $\nu_p(C_k)$ for a given $k$ with $k \leq d^{n+1}m_s$, we examine the $p$-adic valuation of the coefficient of $x^k$ in the expansion of the term:
\begin{equation} \label{e2.6}
B_j \left(\sum_{i=0}^{d} A_i x^i \right)^j
\end{equation}
for each $j$ in the range $0 \leq j \leq d^n e$.

We use arguments similar to those used for obtaining Equations \eqref{e2.3} and \eqref{e2.4} to conclude the following:
\begin{align}
\nu_p(B_j) &\geq \frac{\lambda_{1}}{d^n} (d^n e - j), \quad \forall j \in \{0,1,2,\ldots, d^n e\}; \label{e2.7} \\
\nu_p(A_i) &\geq \lambda (d - i), \quad \forall i \in \{0,1,2,\ldots,d\}. \label{e2.8}
\end{align}

\noindent Now, to prove the lemma, fix an integer $s$ within the range $0 \leq s < t$.

\noindent \textbf{Case 1:} We first prove the result for $k$ in the range $d^{n+1}m_{s+1} \leq k < d^{n+1}m_s$, i.e., $d^n m_{s+1} \leq \frac{k}{d} < d^n m_s$. We split this case into two subcases according to whether $j \leq \frac{k}{d}$ or $j > \frac{k}{d}$.

\noindent \textbf{Subcase (i):} Suppose $j \leq \frac{k}{d} < d^n m_s$, i.e., $d^n e - j > d^n e - d^n m_s = d^n (e - m_s)$. In this situation, the point $(d^n e - j, \nu_p(B_j))$ lies strictly beyond the point $(d^n (e - m_s), r_s)$ as illustrated in Figure \ref{NP of gofn}. Consequently, the slope of the segment joining $(d^n e - j, \nu_p(B_j))$ and $(d^n (e - m_s), r_s)$ is greater than or equal to $\lambda_{s+1}/d^n$, depicted by the blue dotted line in Figure \ref{NP of gofn}. Using this observation for $j<d^nm_s$ and the value of $\lambda_{s+1}$, we deduce
\begin{align*}
    \frac{\nu_p(B_j) - r_s}{d^n m_s - j} &\geq \frac{\lambda_{s+1}}{d^n} = \frac{r_{s+1} - r_s}{d^n (m_s - m_{s+1})}.
\end{align*}
Although the above inequality holds only for $j<d^nm_s$, the next inequlity deduced from above is true for all $j \leq d^nm_s$, 
\begin{align*}
    \nu_p(B_j) - r_s &\geq \frac{r_{s+1} - r_s}{d^n (m_s - m_{s+1})} (d^n m_s - j) \\
    &= r_{s+1} - r_s + \frac{r_{s+1} - r_s}{d^n (m_s - m_{s+1})} (d^n m_{s+1} - j). 
\end{align*}
This implies 
\begin{equation} \label{e2.9}
	\nu_p(B_j) \geq r_{s+1} + \frac{\lambda_{s+1}}{d^n} (d^n m_{s+1} - j)
\end{equation} 
Using this in \eqref{e2.6} and keeping in mind that $\nu_p(A_i) \geq 0$ for all $i \in \{0,1,\ldots,d\}$, we obtain
\[
\nu_p(C_k) \geq \nu_p(B_j) \geq r_{s+1} + \frac{\lambda_{s+1}}{d^n} (d^n m_{s+1} - j).
\]
Further, using $j \leq \frac{k}{d}$, we conclude
\[
\nu_p(C_k) \geq r_{s+1} + \frac{\lambda_{s+1}}{d^n} \left( d^n m_{s+1} - \frac{k}{d} \right) = r_{s+1} + \lambda_{s+1} \left( m_{s+1} - \frac{k}{d^{n+1}} \right).
\]
Therefore, our lemma is proved whenever $j \leq \frac{k}{d}$. \vspace{0.25cm}

\noindent \textbf{Subcase (ii):} Suppose $j > \frac{k}{d}$, and let $k = i_1 + i_2 + \cdots + i_j$ be a partition of $k$ into $j$ terms, where $i_{\ell} \leq d$ for $1 \leq \ell \leq j$. Then there exists a term of degree $k$ in the form $B_j \prod_{\ell=1}^{j} A_{i_{\ell}} x^{i_{\ell}} = \left( B_j \prod_{\ell=1}^{j} A_{i_{\ell}} \right) x^k$ within the expansion of \eqref{e2.6}. We use arguments similar to those used for \eqref{e2.5} to arrive at the following conclusion:
\begin{align*}
    \nu_p(C_k) &\geq \lambda_1 \left[ e - \frac{j}{d^n} + dj - k + \frac{k}{d^{n+1}} - m_{s+1} \right] - \lambda_1 \left( \frac{k}{d^{n+1}} - m_{s+1} \right) \\
    &\geq \lambda_1 \left[ e + \left( d - \frac{1}{d^n} \right) \left(j - \frac{k}{d}\right) - m_{s+1} \right] - \lambda_{s+1} \left( \frac{k}{d^{n+1}} - m_{s+1} \right),
\end{align*}
where the last inequality holds because for any $s > 0$, as $\frac{k}{d^{n+1}} - m_{s+1} \geq 0$ and $\lambda_1 \leq \lambda_{s+1}$. Additionally, if $j - \frac{k}{d} \geq 1$, we use $d \geq 1$ to derive
\[
\nu_p(C_k) \geq \lambda_1 \left[ e + (d - 1) - m_{s+1} \right] + \lambda_{s+1} \left( m_{s+1} - \frac{k}{d^{n+1}} \right).
\]
Using the given fact that $r_t \leq \lambda_1 (d + e - 1)$ in Lemma \ref{lemma2.1} with $\alpha = e + d - 1$, the above inequality becomes
\[
\nu_p(C_k) \geq r_{s+1} + \lambda_{s+1} \left( m_{s+1} - \frac{k}{d^{n+1}} \right).
\]

\noindent Therefore, our lemma holds whenever $j - \frac{k}{d} \geq 1$. The only remaining part in this subcase is when $0 < j - \frac{k}{d} < 1$, which implies $k < dj < d + k$. Thus, there exists an integer $\gamma$ in the range $0 < \gamma < d$ satisfying $k + \gamma = dj$. Further, we use $k = i_1 + i_2 + \ldots + i_j$ to deduce $\left( \sum_{\ell=1}^{j} i_{\ell} \right) + \gamma = dj$. This implies $\gamma = \sum_{\ell=1}^{j} (d - i_{\ell})$. Additionally, considering the condition $j < \frac{k}{d} + 1$ and bearing in mind that $k < d^{n+1} m_s$, we infer that $j < d^n m_s + 1$, i.e., $j \leq d^n m_s$. Now, we use reasoning similar to those used for \eqref{e2.9} to deduce $\nu_p(B_j) \geq r_{s+1} + \frac{\lambda_{s+1}}{d^n} (d^n m_{s+1} - j)$. Using this inequality and \eqref{e2.8}, we obtain
\[
\nu_p(C_k) \geq \nu_p(B_j) + \sum_{\ell=1}^{j} \nu_p(A_{i_{\ell}})
\geq r_{s+1} + \frac{\lambda_{s+1}}{d^n} (d^n m_{s+1} - j) + \lambda \sum_{\ell=1}^{j} (d - i_{\ell}).
\]

	\noindent Substituting \( j = \frac{k}{d} + \frac{\gamma}{d} \) and using \( \gamma = \sum_{\ell=1}^{j} (d-i_{\ell}) \), we have
\begin{align*}
\nu_p(C_k) &\geq r_{s+1} + \frac{\lambda_{s+1}}{d^n} \left( d^n m_{s+1} - \left(\frac{k}{d} + \frac{\gamma}{d}\right) \right) + \lambda \gamma \\
&\geq r_{s+1} + \lambda_{s+1} \left( m_{s+1} - \frac{k}{d^{n+1}} \right) - \frac{\lambda_{s+1}}{d^n} \frac{\gamma}{d} + \lambda_1 \gamma,
\end{align*}
where the last inequality follows from the hypothesis \( \lambda \geq \lambda_1 \). Therefore, to establish our lemma in this subcase, it is enough to show that
\[
\lambda_1 \gamma - \frac{\lambda_{s+1}}{d^n} \frac{\gamma}{d} \geq 0.
\]
Taking into account the conditions \( \gamma > 0 \), \( n \geq 0 \), and \( \lambda_t \geq \lambda_{s+1} \), for all $s$ with $0< s \leq t$, the above inequality holds true if
\[
\lambda_1 \geq \frac{\lambda_{t}}{d} = \frac{1}{d} \frac{r_t - r_{t-1}}{m_{t-1} - m_t}, \text{ i.e., } r_t \leq \lambda_1 d(m_{t-1} - m_t) + r_{t-1}.
\]
Furthermore, noting that \( \frac{r_{t-1}}{e - m_{t-1}} \geq \lambda_1 \) and keeping in mind that \( m_t = 0 \), the inequality above holds true if
\[
r_t \leq \lambda_1 d m_{t-1} + \lambda_1 (e - m_{t-1}), \text{ i.e., } r_t \leq \lambda_1 [e + (d - 1)m_{t-1}].
\]
The last inequality follows from \( m_{t-1} \geq 1 \) and the hypothesis that \( r_t \leq \lambda_{1} (d + e - 1) \). This proves Subcase (ii). Thus, our lemma holds for any \( s \), \( 0 \leq s < t \), whenever \( k \) lies in the range \( d^{n+1} m_{s+1} \leq k < d^{n+1} m_s \). This completes the proof of Case 1. \vspace{0.5cm}

\noindent \textbf{Case 2:} Now suppose that \( k < d^{n+1} m_s \). Note that the sequence \( m_i \)'s are decreasing; hence there exists an integer \( \alpha < t \) such that \( \alpha \geq s \) and \( d^{n+1} m_{\alpha+1} \leq k < d^{n+1} m_{\alpha} \). Using exactly similar arguments to Case 1 for \( \alpha \), we see that
\begin{align}
\nu_p(C_k) &\geq r_{\alpha+1} +  \lambda_{\alpha+1} \left(m_{\alpha+1} - \frac{k}{d^{n+1}}\right) \label{e2.10} \\
&= r_{\alpha} + r_{\alpha+1} - r_{\alpha} + \frac{r_{\alpha+1} - r_{\alpha}}{m_{\alpha} - m_{\alpha+1}} \left(m_{\alpha+1} - \frac{k}{d^{n+1}}\right) \nonumber \\
&= r_{\alpha} + \frac{r_{\alpha+1} - r_{\alpha}}{m_{\alpha} - m_{\alpha+1}} \left(m_{\alpha} - m_{\alpha+1} + m_{\alpha+1} - \frac{k}{d^{n+1}}\right) \nonumber \\
&= r_{\alpha} +  \lambda_{\alpha+1} \left(m_{\alpha} - \frac{k}{d^{n+1}}\right) \nonumber \\
&> r_{\alpha} +  \lambda_{\alpha} \left(m_{\alpha} - \frac{k}{d^{n+1}}\right), \label{e2.11}
\end{align}
where the last inequality holds because \( \lambda_{\alpha+1} > \lambda_{\alpha} \) and \( k < d^{n+1} m_{\alpha} \). If \( \alpha = s \), then the lemma is clearly true using \eqref{e2.10}. If \( \alpha = s + 1 \), we use \eqref{e2.11} to prove our lemma. Now, assuming \( \alpha > s + 1 \) and keeping in mind \( \frac{k}{d^{n+1}} < m_{\alpha} \), we apply Lemma \ref{lemma2.2} (i) to the right-hand side of \eqref{e2.11} to conclude 
\[
\nu_p(C_k) > r_{s+1} +  \lambda_{s+1} \left(m_{s+1} - \frac{k}{d^{n+1}}\right).
\]
This completes the proof of Case 2. Thus, our lemma is proved for \( s \). Since \( s \) was arbitrarily chosen from the range \( 0 \leq s < t \), our result holds for all \( s \), \( 0 \leq s < t \). This completes the proof of the lemma.
\end{proof}

\section{Proof of Theorem \ref{main thm1}} \label{sec3}
\begin{proof}[Proof of Theorem \ref{main thm1}:]
    Observe that $0 = m_t < m_{t-1} < \cdots < m_1 < m_0 = e$. Let $\nu_p(b_{m_i})$ be denoted by $r_i$ for $0 \leq i \leq t$. Hence, $\lambda_i$ can be expressed as 
    \[
    \lambda_i = \frac{r_i - r_{i-1}}{m_{i-1} - m_i} \quad \text{for} \quad 1 \leq i \leq t.
    \]
    We proceed by induction on $n \geq 0$. For $n = 0$, we have $g \circ f^0 = g$, and the result is trivially true.

    Now, fix an integer $n > 0$. Assume that the result holds for $n$, i.e., the successive vertices of $NP_p(g \circ f^n)$ are given by the set
    \[
    \{(0,0), (d^n(e - m_1), r_1), \dots, (d^n(e - m_{t-1}), r_{t-1}), (d^n e, r_t)\}.
    \]
    
    Our goal is to show that $NP_p(g \circ f^{n+1})$, has the structure as depicted in Figure \ref{NP of gofn+1}. Specifically, we need to show that the successive vertices of $NP_p(g \circ f^{n+1})$ are given by the set
    \[
    \{(0,0), (d^{n+1}(e - m_1), r_1), \dots, (d^{n+1}(e - m_{t-1}), r_{t-1}), (d^{n+1} e, r_t)\}.
    \]
    
    \begin{figure}[htbp]
        \centering
        \begin{tikzpicture}[scale=1.2]
            \draw[->, thick] (0,-0.5) -- (0,6);
            \draw[->, thick] (-0.5,0) -- (9.5,0);
            
            \coordinate (A) at (0,0);
            \coordinate (B) at (2,0.15);
            \coordinate (C) at (3.7,0.6);
            \coordinate (D) at (5.2,1.6);
            \coordinate (E) at (6.2,2.8);
            \coordinate (F) at (7,5);
            
            \coordinate (I) at (4.3,2.7);
            
            \draw[thick] (A) -- (B);
            \draw[dotted, thick, black!60] (B) -- (C);
            \draw[thick] (C) -- (D) node[midway, above] {$\frac{\lambda_{s+1}}{d^{n+1}}$};
            \draw[dotted, thick, black!60] (D) -- (E);
            \draw[thick] (E) -- (F) node[midway, right] {$\frac{\lambda_t}{d^{n+1}}$};
            
            \draw[dotted, thick, red] (B) -- ++(7.5,0.6);
            \draw[dotted, thick, green] (A) -- (I);
            \draw[dotted, thick, blue] (C) -- (I);
            
            \foreach \point/\position/\name in {
                A/below left/{(0,0)}, 
                B/below/{$(d^{n+1}(e - m_1), r_1)$}, 
                C/right/{$(d^{n+1}(e - m_{s-1}), r_{s-1})$}, 
                D/right/{$(d^{n+1}(e - m_s), r_s)$}, 
                E/right/{$(d^{n+1}(e - m_{t-1}), r_{t-1})$}, 
                F/above/{$(d^{n+1} e, r_t)$}, 
                I/above/{$(d^{n+1} e - k, \nu_p(C_k))$}
            } {
                \fill (\point) circle (2pt) node[\position] {\name};
            }
        \end{tikzpicture}
        \caption{Newton polygon of $g \circ f^{n+1}$ with respect to $p$} \label{NP of gofn+1}
    \end{figure}
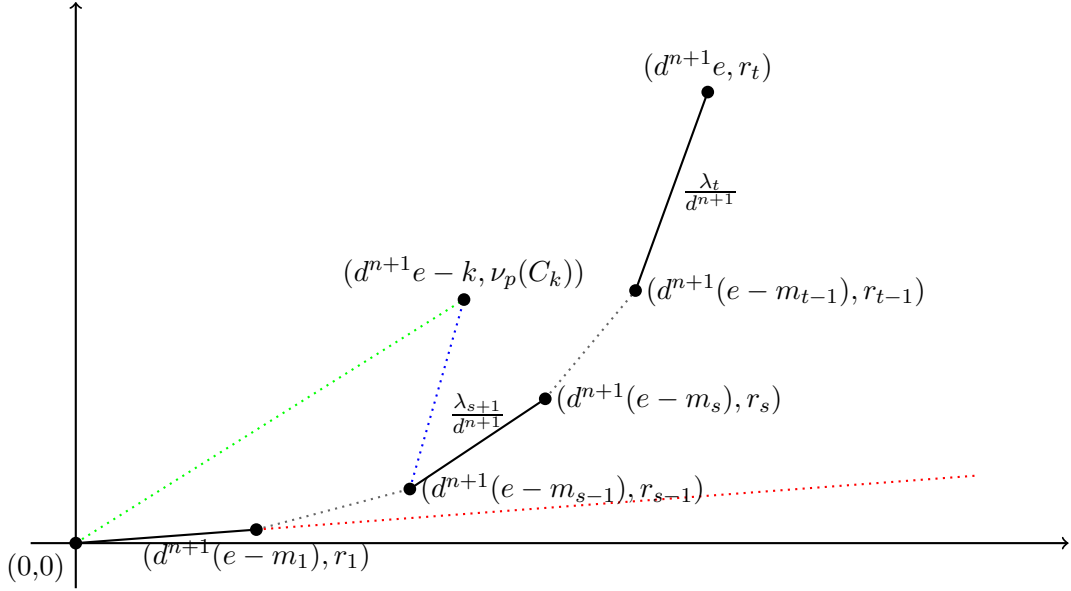

\noindent Write $f(x) = \sum_{i=0}^{d} A_i x^i$ and $(g \circ f^n)(x) = \sum_{j=0}^{d^n e} B_j x^j$. Thus, we have 
\[
(g \circ f^{n+1})(x) = \sum_{j=0}^{d^n e} B_j \left( \sum_{i=0}^{d} A_i x^i \right)^j.
\]
Let $C_k$ denote the coefficient of $x^k$ such that 
\[
(g \circ f^{n+1})(x) = \sum_{k=0}^{d^{n+1}e} C_k x^k.
\]

Using Lemmas \ref{lemma2.3} and \ref{lemma2.4}, we will show that the first edge of $NP_p(g \circ f^{n+1})$ has endpoints $(0,0)$ and $(d^{n+1}(e - m_1), r_1)$. For $s=0$, we have $m_0 = e$. By applying Lemma \ref{lemma2.3} with $s = 0$, we obtain $\nu_p(C_{d^{n+1} e}) = r_0 = 0$, meaning that the first edge of $NP_p(g \circ f^{n+1})$ starts at the point $(0,0)$. 

Next, applying Lemma \ref{lemma2.4} with $s=0$ and $k \leq d^{n+1} m_0 = d^{n+1} e$, we deduce that
\begin{align*}
    \nu_p(C_k) &\geq r_1 + \lambda_1 \left( m_1 - \frac{k}{d^{n+1}} \right) \\
    &= r_0 + r_1 - r_0 + \frac{r_1 - r_0}{e - m_1} \left( m_1 - \frac{k}{d^{n+1}} \right) \\
    &= r_0 + \frac{r_1 - r_0}{e - m_1} \left( e - \frac{k}{d^{n+1}} \right).
\end{align*}

\noindent Thus, we have
\begin{equation}\label{e3.1}
    \frac{\nu_p(C_k) - r_0}{d^{n+1} e - k} \geq \frac{1}{d^{n+1}} \cdot \frac{r_1 - r_0}{e - m_1} = \frac{\lambda_1}{d^{n+1}}.
\end{equation}

Hence, every point $(d^{n+1} e - k, \nu_p(C_k))$ lies on or above the line segment joining $(0,0)$ and $(d^{n+1}(e - m_1), r_1)$. By applying Lemma \ref{lemma2.3} for $s = 1$, we obtain $\nu_p(C_{d^{n+1} m_1}) = r_1$. To show that $(d^{n+1}(e - m_1), \nu_p(C_{d^{n+1} m_1}))$ is the endpoint of the first edge, it suffices to prove that any point $(d^{n+1} e - k, \nu_p(C_k))$ beyond $(d^{n+1}(e - m_1), \nu_p(C_{d^{n+1} m_1}))$ lies strictly above the line segment representing the first edge, as indicated by the green dotted segment in Figure \ref{NP of gofn+1}.

Now, consider a point $(d^{n+1} e - k, \nu_p(C_k))$ lying beyond $(d^{n+1}(e - m_1), r_1)$, i.e., $d^{n+1} e - k > d^{n+1}(e - m_1)$. Then, we have $k < d^{n+1} m_1$. Lemma \ref{lemma2.4} with $s = 1$ gives

\begin{align*}
	\nu_p(C_k) &\geq r_2 + \lambda_2 \left( m_2 - \frac{k}{d^{n+1}} \right) \\
	&= r_1 + r_2 - r_1 + \frac{r_2 - r_1}{m_1 - m_2}\left( m_2 - \frac{k}{d^{n+1}} \right) \\
	&= r_1 + \frac{r_2 - r_1}{m_1 - m_2}\left( m_1 - \frac{k}{d^{n+1}} \right) \\
	&> r_1 + \lambda_1 \left( m_1 - \frac{k}{d^{n+1}} \right).
\end{align*}
where the last inequality follows using the fact that $k < d^{n+1}m_1$ and $\lambda_{1} < \lambda_{2}$. Following similar reasoning as for inequality \eqref{e3.1}, we get
\[
    \frac{\nu_p(C_k) - r_0}{d^{n+1} e - k} > \frac{\lambda_1}{d^{n+1}}.
\]
Thus, every point $(d^{n+1} e - k, \nu_p(C_k))$ lying beyond $(d^{n+1}(e - m_1), r_1)$ lies strictly above the line segment joining $(0,0)$ and $(d^{n+1}(e - m_1), r_1)$. Therefore, we conclude that the first edge of $NP_p(g \circ f^{n+1})$ has endpoints $(0,0)$ and $(d^{n+1}(e - m_1), r_1)$.

For any $s$ in the range $0 < s \leq t$, similar arguments will show that the $s$-th edge of $NP_p(g \circ f^{n+1})$ has endpoints $(d^{n+1}(e - m_{s-1}), r_{s-1})$ and $(d^{n+1}(e - m_s), r_s)$. This completes the proof of the theorem.
\end{proof}

We now present a few examples to illustrate that the assumptions made in Theorem \ref{main thm1} are not only sufficient but also strictly necessary for its conclusions. For a given polynomial \( g(x) \in \mathbb{Q}[x] \) of degree \( e \), we use the notation:
\[
NP_p(g): (0,0) \rightarrow (e-m_1, \nu_p(b_{m_1})) \rightarrow \cdots \rightarrow (e, \nu_p(b_0))
\]
to represent the successive vertices of the Newton polygon of \( g(x) \) with respect to the prime \( p \). These vertices are given by the set 
\[
\{ (0,0), (e-m_1, \nu_p(b_{m_1})), \dots, (e, \nu_p(b_0)) \}.
\]

In what follows, Example \ref{eg1} illustrates the importance of the condition \( \lambda \geq \lambda_1 \), while Examples \ref{eg2} and \ref{eg3} highlights the necessity of the condition \( r_t \leq \lambda_1(d + e - 1) \).

\begin{example} \label{eg1}
    Let \( f(x) = x^3 + 2x + 4 \) and \( g(x) = x^3 + 4x + 2^4 \). Then 
    \[
    (g \circ f)(x) = x^9 + 6x^7 + 12x^6 + 12x^5 + 48x^4 + 60x^3 + 48x^2 + 104x + 96.
    \]
    Furthermore, the Newton polygons of \( f \), \( g \), and \( g \circ f \) with respect to the prime \( 2 \) are:
    \begin{align*}
        NP_2(f) &: (0,0) \rightarrow (2,1) \rightarrow (3,2), \\
        NP_2(g) &: (0,0) \rightarrow (2,2) \rightarrow (3,4), \\
        NP_2(g \circ f) &: (0,0) \rightarrow (6,2) \rightarrow (8,3) \rightarrow (9,5).
    \end{align*}
    Using the notations in Theorem \ref{main thm1}, we observe that \( t=2 \), \( d=e=3 \), \( \lambda = \frac{1}{2} < 1 = \lambda_1 \), and \( r_t = 4 < 5 = \lambda_1(d + e - 1) \). While all other conditions of Theorem \ref{main thm1} are satisfied, the condition \( \lambda \geq \lambda_1 \) is violated, leading to the failure of Newton polygon structure preservation. This underscores the necessity of the condition \( \lambda \geq \lambda_1 \).
\end{example}

Next two examples demonstrate the importance of the condition \( r_t \leq \lambda_1(d + e - 1) \) for cases where \( t=2 \) and \( t=3 \).

\begin{example} \label{eg2}
    Let \( f(x) = x^3 + 2x^2 + 2x + 4 \) and \( g(x) = x^3 + 2x + 8 \). Then 
    \[
    (g \circ f)(x) = x^9 + 6x^8 + 18x^7 + 44x^6 + 84x^5 + 120x^4 + 154x^3 + 148x^2 + 100x + 80.
    \]
    The Newton polygons of \( f \), \( g \), and \( g \circ f \) with respect to the prime \( 2 \) are:
    \begin{align*}
        NP_2(f) &: (0,0) \rightarrow (2,1) \rightarrow (3,2), \\
        NP_2(g) &: (0,0) \rightarrow (2,1) \rightarrow (3,3), \\
        NP_2(g \circ f) &: (0,0) \rightarrow (6,1) \rightarrow (8,2) \rightarrow (9,4).
    \end{align*}
    Using the notations in Theorem \ref{main thm1}, we have \( t=2 \), \( d=e=3 \), \( \lambda=\lambda_{1} = \frac{1}{2} \), and \( r_t = 3 > \frac{5}{2} = \lambda_1(d + e - 1) \). This example illustrates the necessity of condition (i) of Theorem \ref{main thm1}.
\end{example}

\begin{example} \label{eg3}
    Consider the case \( t = 3 \). Let \( f(x) = x^{11} + 2x^4 + 4x + 16 \). The Newton polygons of \( f \) and \( f^2 \) with respect to the prime \( 2 \) are given by:
    \begin{align*}
        NP_2(f) &: (0,0) \rightarrow (7,1) \rightarrow (10,2) \rightarrow (11,4), \\
        NP_2(f^2) &: (0,0) \rightarrow (77,1) \rightarrow (110,2) \rightarrow (117,3) \rightarrow (121,4).
    \end{align*}
    Using the notations of Theorem \ref{main thm1}, we observe that in this example \( t = 3 \), \( d = e = 3 \), \( \lambda = \lambda_1 = \frac{1}{7} \), and \( r_3 = 4 > 3 = \lambda_1(d + e - 1) \). This demonstrates the necessity of the condition \( r_t \leq \lambda_1(d + e - 1) \) in Theorem \ref{main thm1}.
\end{example}

\section{Proof of Theorem \ref{main thm2}} \label{sec4}
\begin{proof}[Proof of Theorem \ref{main thm2}:]
    Denote $\nu_p(b_{m_i})$ by $r_i$ for $0 \leq i \leq t$. Thus, $\lambda_i$ can be written as $\lambda_i = \frac{r_i - r_{i-1}}{m_{i-1} - m_i}$ for $1 \leq i \leq t$. Our goal is to show that the successive vertices of $NP_p(g \circ f)$ are given by the set
    \[
    \{(0, r_0), (d(e - m_1), r_1), \dots, (d(e - m_{t-1}), r_{t-1}), (de, r_t)\}.
    \]
    
    \noindent Write $f(x) = \sum_{i=0}^{d} a_i x^i$. So, we have
    \[
    (g \circ f)(x) = \sum_{j=0}^{e} b_j \left( \sum_{i=0}^{d} a_i x^i \right)^j.
    \]
    Let $c_k$ denote the coefficients of $x^k$ in this expansion, so that
    \[
    (g \circ f)(x) = \sum_{k=0}^{de} c_k x^k.
    \]
    The proof of our theorem will proceed similarly to the proof of Theorem \ref{main thm1}, once we establish the following two key equations:
    
    \begin{equation}\label{(I)}
        \text{For } k = d m_s \text{ with } 0 \leq s \leq t, \text{ we have } \nu_p(c_k) = r_s,
    \end{equation}
    
    \begin{equation}\label{(II)}
        \text{for } k \leq d m_s \text{ with } 0 \leq s \leq t, \text{ we have } \nu_p(c_k) \geq r_{s+1} + \lambda_{s+1} \left( m_{s+1} - \frac{k}{d} \right).
    \end{equation}
    
    To prove Equations \eqref{(I)} and \eqref{(II)}, we need to determine $\nu_p(c_k)$ for a given $k$ within the range $0 \leq k \leq de$. That is, we need the $p$-adic valuation of the coefficient of $x^k$ in the expansion of the term:
    
    \begin{equation}\label{e4.3}
        b_j \left( \sum_{i=0}^{d} a_i x^i \right)^j
    \end{equation}
    for each $j$ in the range $0 \leq j \leq e$. By hypothesis, the successive vertices of $NP_p(g)$ are given by the set
    \[
    \{ (0, r_0), (e - m_1, r_1), \dots, (e - m_{t-1}, r_{t-1}), (e, r_t) \}.
    \]
    Therefore, $NP_p(g)$ is as shown in Figure \ref{NP of g}.
    
    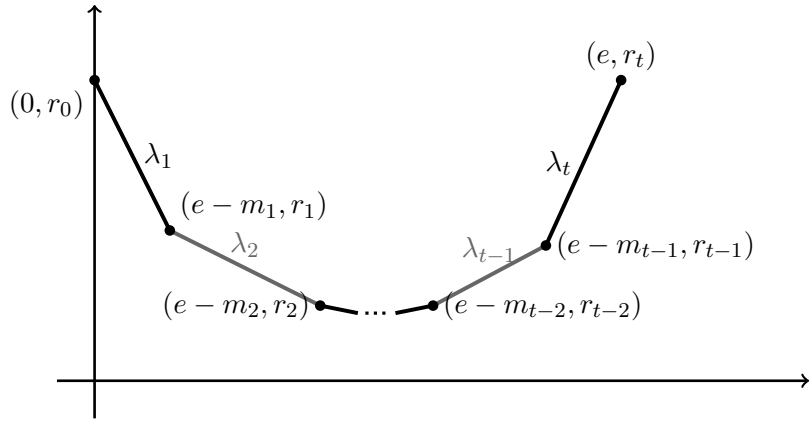
\begin{figure}[h]
        \centering
        \begin{tikzpicture}
            \draw[->, line width=1pt] (0,-0.5) -- (0,5);
            \draw[->, line width=1pt] (-0.5,0) -- (9.5,0);
            
            \coordinate (A) at (0,4);
            \coordinate (B) at (1,2);
            \coordinate (C) at (3,1);
            \coordinate (C1) at (3.5,0.9);
            \coordinate (C2) at (3.6,0.9);
            \coordinate (C3) at (3.9,0.9);
            \coordinate (C4) at (4,0.9);
            \coordinate (D) at (4.5,1);
            \coordinate (E) at (6,1.8);
            \coordinate (F) at (7,4);
            
            \draw[line width=1.5pt] (A) -- (B) node[midway, right] {$\lambda_{1}$};
            \draw[line width=1.5pt, black!60] (B) -- (C) node[midway, above] {$\lambda_{2}$};
            \draw[line width=1.5pt] (C) -- (C1);
            \draw[dotted, line width=1.2pt] (C2) -- (C3);
            \draw[line width=1.5pt] (C4) -- (D);
            \draw[line width=1.5pt, black!60] (D) -- (E) node[midway, above] {$\lambda_{t-1}$};
            \draw[line width=1.5pt] (E) -- (F) node[midway, left] {$\lambda_t$};
            
            \foreach \point/\position/\name in {
                A/below left/{$(0, r_0)$}, 
                B/above right/{$(e - m_1, r_1)$}, 
                C/left/{$(e - m_2, r_2)$}, 
                D/right/{$(e - m_{t-2}, r_{t-2})$}, 
                E/right/{$(e - m_{t-1}, r_{t-1})$}, 
                F/above/{$(e, r_t)$}
            } {
                \fill (\point) circle (2pt) node[\position] {\name};
            }
        \end{tikzpicture}
        \caption{Newton polygon of $g(x)$ with respect to $p$}
        \label{NP of g}
    \end{figure}

	{}
	
	\noindent Recall the hypothesis that $ \nu_p(a_d) = 0 $ and 
\begin{align}
    \nu_p(a_i) \geq \frac{u}{\beta} (d-i) \quad \text{for all } i \in \{0, 1, 2, \ldots, d\}. \label{e4.5}
\end{align}
\vspace{0.5cm}

\noindent We would like to point out that the maximum possible denominator of the slope of any edge of Newton polygon cannot exceed degree of the polynomial, i.e., in our case $\beta \leq d$. We now prove \eqref{(I)}. The proof of \eqref{(I)} is split into two cases depending on whether $s = 0$ or $0 < s \leq t$. \\

\noindent {\bf Case $s = 0$:} For $s = 0$, we have $k = dm_0 = de$. In this situation, the only term of degree $k$ in the expression of $(g \circ f)(x)$ is $b_{e} a_d^e$. Hence, we have:
\[
    \nu_p(c_k) = \nu_p(b_e) + e \nu_p(a_d) = r_0 + e(0) = r_0.
\]
This proves \eqref{(I)} for $s = 0$. \vspace{0.25cm}

\noindent {\bf Case $0 < s \leq t$:} Fix an integer $s$ with $0 < s \leq t$ and let $k = dm_s$. If $j < \frac{k}{d}$, then the highest possible degree of $x$ in the expansion of \eqref{e4.3} is $dj$, which is strictly less than $k$. Therefore, no term of degree $k$ exists in \eqref{e4.3} whenever $j < \frac{k}{d}$.

Now assume $j > \frac{k}{d} = m_s$. Recall that the $m_i$'s form a decreasing sequence, so there exists an integer $\alpha$ in the range $0 < \alpha \leq s$ such that $m_{\alpha} < j \leq m_{\alpha-1}$. In this case, we have $e - j \geq e - m_{\alpha-1}$, i.e., the point $(e - j, \nu_p(b_j))$ lies beyond $(e - m_{\alpha-1}, r_{\alpha-1})$. Thus, we obtain the inequality: 
\begin{align} \label{e4.6}
    \nu_p(b_j) - r_{\alpha-1} &\geq \lambda_{\alpha} (m_{\alpha-1} - j), \nonumber \\
    \text{i.e.,} \quad \nu_p(b_j) &\geq r_{\alpha-1} + \lambda_{\alpha} (m_{\alpha-1} - j).
\end{align}

Suppose that $k$ can be partitioned into $j$ terms as $k = i_1 + i_2 + \cdots + i_j$ with each $i_{\ell} \leq d$ for $1 \leq \ell \leq j$. Then, there exists a term of the form
\[
    b_j \prod_{\ell=1}^{j} a_{i_{\ell}} x^{i_{\ell}} =  \left( b_j \prod_{\ell=1}^{j} a_{i_{\ell}} \right) x^k
\]
of degree $k$ in the expansion of \eqref{e4.3}. Our goal is to show that for each such partition of $k$ with $j > \frac{k}{d}$, the inequality $\nu_p(c_k) > r_s$ holds.
Using \eqref{e4.5} and \eqref{e4.6}, we get:
\[
    \nu_p(c_k) \geq \nu_p\left(b_j \prod_{\ell=1}^{j} a_{i_{\ell}}\right) = \nu_p(b_j) + \sum_{\ell=1}^{j} \nu_p(a_{i_{\ell}}) \geq r_{\alpha-1} + \lambda_{\alpha} (m_{\alpha-1} - j) + \sum_{\ell=1}^{j} \frac{u}{\beta} (d-i_{\ell}).
\]
Furthermore, applying Lemma \ref{lemma2.2} (ii), to prove $\nu_p(c_k) > r_s$, it suffices to show that:
	\begin{align*}
    r_{\alpha-1} + \lambda_{\alpha}(m_{\alpha-1} - j) + \frac{u}{\beta} (dj-k) &> r_s = r_{\alpha-1} + \sum_{i=\alpha}^{s} \lambda_i (m_{i-1} - m_i), \\
    \text{i.e.,} \quad \lambda_{\alpha}(m_{\alpha-1} - j) + \frac{u}{\beta} (dj-dm_{s}) &> \lambda_{\alpha} (m_{\alpha-1} - m_{\alpha}) + \sum_{i=\alpha+1}^{s} \lambda_i (m_{i-1} - m_i), \\
    \text{i.e.,} \quad \frac{u}{\beta} d(j-m_{s}) &> \lambda_{\alpha}(j - m_{\alpha}) + \sum_{i=\alpha+1}^{s} \lambda_i (m_{i-1} - m_i).
\end{align*}

On the contrary suppose that the above inequality does not hold. This means we have:
\[
    \frac{d}{\beta} (j-m_{s}) \leq \frac{\lambda_{\alpha}}{u}(j - m_{\alpha}) + \sum_{i=\alpha+1}^{s} \frac{\lambda_i}{u}(m_{i-1} - m_i) < j - m_{\alpha} + (m_{\alpha} - m_s),
\]
where the last inequality follows from the fact that $\lambda_i \leq |\lambda_i| < u$ for all $i$ in the range $0 < i \leq t$, along with $j > m_{\alpha}$ and the fact that the $m_i$'s form a decreasing sequence. 
Since $j-m_s>0$, the above inequality is equivalent to $\frac{d}{\beta}<1$, which is a contradiction to the fact that $\beta \leq d$. Hence, the desired inequality holds, implying $\nu_p(c_k) > r_s$ whenever $j > \frac{k}{d}$. 

The only remaining case is $j = m_s$, for which there is only one term of degree $k$ in \eqref{e4.3}, namely $b_j a_d^j$. Its $p$-adic valuation is:
\[
    \nu_p(b_j a_d^j) = \nu_p(b_{m_s}) + j \nu_p(a_d) = r_s.
\]
Combining all of these results, we obtain $\nu_p(c_{dm_s}) = r_s$. Since $s$ was chosen arbitrarily from the range $0 < s \leq t$, this completes the proof of \eqref{(I)} for all $0 < s \leq t$. \vspace{0.25cm}

\noindent We now proceed to prove \eqref{(II)}. Fix an integer $s$ in the range $0 \leq s \leq t$.

\noindent \textbf{Case 1:} We first establish \eqref{(II)} for $k$ such that $dm_{s+1} < k \leq dm_s$. 

Note that if $j < \frac{k}{d}$, the highest possible degree of $x$ in the expansion of \eqref{e4.3} is $dj$, which is strictly smaller than $k$. Hence, no term of degree $k$ exists in the expansion of \eqref{e4.3} whenever $j < \frac{k}{d}$. 

Now, assume $j \geq \frac{k}{d} > m_{s+1}$. Recall that the $m_i$'s form a decreasing sequence, so there exists an integer $\alpha$ in the range $0 < \alpha \leq s+1$ such that $m_{\alpha} < j \leq m_{\alpha-1}$. Suppose that $k$ can be partitioned into $j$ terms as $k = i_1 + i_2 + \cdots + i_j$, where each $i_{\ell} \leq d$ for $1 \leq \ell \leq j$. Then, there exists a term in the form:
\[
    b_j \prod_{\ell=1}^{j} a_{i_{\ell}} x^{i_{\ell}} = \left( b_j \prod_{\ell=1}^{j} a_{i_{\ell}} \right) x^k
\]
of degree $k$ in the expansion of \eqref{e4.3}. Now, we will show that for each such partition of $k$ with $j \geq \frac{k}{d}$, the inequality \eqref{(II)} holds. Using \eqref{e4.5} and arguments similar to those used for \eqref{e4.6}, we obtain:
\[
    \nu_p(c_k) \geq \nu_p\left( b_j \prod_{\ell=1}^{j} a_{i_{\ell}} \right) = \nu_p(b_j) + \sum_{\ell=1}^{j} \nu_p(a_{i_{\ell}}) \geq r_{\alpha-1} + \lambda_{\alpha}(m_{\alpha-1} - j) + \sum_{\ell=1}^{j} \frac{u}{\beta} (d-i_{\ell}).
\]
Recall that $\alpha - 1 < \alpha \leq s+1$. By applying Lemma \ref{lemma2.2} (ii) to $r_{s+1}$, we see that in order to prove \eqref{(II)}, it suffices to show that
\begin{align}
    r_{\alpha-1} + \lambda_{\alpha}(m_{\alpha-1} - j) + \frac{u}{\beta} (dj-k) &\geq r_{\alpha-1} + \sum_{i=\alpha}^{s+1} \lambda_i (m_{i-1} - m_i) + \lambda_{s+1} \left(m_{s+1} - \frac{k}{d}\right), \nonumber \\
    \text{i.e.,} \quad \frac{u}{\beta} d\left(j-\frac{k}{d}\right) &\geq \lambda_{\alpha}(j - m_{\alpha}) + \sum_{i=\alpha+1}^{s+1} \lambda_i (m_{i-1} - m_i) + \lambda_{s+1} \left(m_{s+1} - \frac{k}{d}\right). \label{e4.7}
\end{align}

\noindent In the case where $\alpha = s+1$, inequality \eqref{e4.7} becomes:
\[
    \frac{u}{\beta} d\left(j-\frac{k}{d}\right) \geq \lambda_{s+1}(j - m_{s+1}) + \lambda_{s+1} \left(m_{s+1} - \frac{k}{d}\right) = \lambda_{s+1} \left(j - \frac{k}{d}\right).
\]
The above inequlity trivially holds for $j=\frac{k}{d}$. Now for $j>\frac{k}{d}$, if the above inequality does not hold, then we use the fact $\lambda_{s+1} \leq |\lambda_{s+1}| < u$ to conclude $\frac{d}{\beta} < 1$, which contradicts $\beta \leq d$. Therefore, inequality \eqref{e4.7} holds for $\alpha = s+1$.

Now suppose $\alpha < s+1$. In this case, inequality \eqref{e4.7} can be rewritten as:
\[
    \frac{u}{\beta} d\left( j-\frac{k}{d} \right) \geq \lambda_{\alpha}(j - m_{\alpha}) + \sum_{i=\alpha+1}^{s} \lambda_i (m_{i-1} - m_i) + \lambda_{s+1} \left(m_{s} - \frac{k}{d}\right).
\]
Assume, for the sake of contradiction, that the last inequality is false. That is, we have:
\begin{align*}
    \frac{d}{\beta} \left(j-\frac{k}{d}\right) & < \frac{\lambda_{\alpha}}{u}(j - m_{\alpha}) + \sum_{i=\alpha+1}^{s} \frac{\lambda_i}{u}(m_{i-1} - m_i) + \frac{\lambda_{s+1}}{u}\left(m_{s} - \frac{k}{d}\right) \\
    &< j - m_{\alpha} + m_{\alpha} - m_s + m_s - \frac{k}{d} = j - \frac{k}{d},
\end{align*}
where the last inequality follows from $\lambda_i \leq |\lambda_i| < u$ for all $i$ in the range $0 < i \leq t$, $j > m_{\alpha}$, $m_s \geq \frac{k}{d}$, and the fact that the $m_i$'s are decreasing. For $j=\frac{k}{d}$, the above inequality in contradictory. So assume $j>\frac{k}{d}$. Since $j-\frac{k}{d} > 0$, the above inequality can be rewritten as $\frac{d}{\beta} <1$, which contradicts $\beta \leq d$. Therefore, \eqref{(II)} holds for $k$ in the range $dm_{s+1} < k \leq dm_s$.

\vspace{.25cm}

\noindent \textbf{Case 2:}
Now consider the case where $k \leq dm_{s+1}$, i.e., $\frac{k}{d} \leq m_{s+1}$. Recall that the $m_i$'s form a decreasing sequence. Therefore, there exists an integer $\alpha$ in the range $s+1 < \alpha \leq t$ such that $m_{\alpha} < \frac{k}{d} \leq m_{\alpha-1}$, i.e., $dm_{\alpha} < k \leq dm_{\alpha-1}$. Using arguments similar to those in Case 1, we obtain:
\begin{align} \label{e4.8}
    \nu_p(c_k) &\geq r_{\alpha} + \lambda_{\alpha}\left(m_{\alpha} - \frac{k}{d}\right) \nonumber \\
    &> r_{\alpha-1} + \lambda_{\alpha-1} \left( m_{\alpha-1} - \frac{k}{d} \right)
\end{align}
where the last inequality can be obtained using the arguments similar to those used to obtain \eqref{e2.11} from \eqref{e2.10}. If $\alpha = s+2$, then \eqref{(II)} holds true from \eqref{e4.8}. So assume $\alpha > s+2$, we apply Lemma \ref{lemma2.2} (i) and use \eqref{e4.8} to deduce that:
\[
    \nu_p(c_k) > r_{s+1} + \lambda_{s+1} \left( m_{s+1} - \frac{k}{d} \right)
\]
for all $dm_{\alpha} < k \leq dm_{\alpha-1}$. This completes the proof of \textbf{Case 2}, and hence, the inequality \eqref{(II)} holds true for all $k \leq dm_s$.

Since $s$ was chosen arbitrarily from the range $0 \leq s < t$, inequality \eqref{(II)} holds for all $s$, $0 \leq s < t$. This concludes the proof of inequality \eqref{(II)}, and consequently, the proof of Theorem \ref{main thm2}.
\end{proof}

We now provide an example to illustrate that the assumption \( u > \lambda_t \) in Theorem \ref{main thm2} is crucial for preserving the structure of the Newton polygon of \( g \) under composition by \( f \). Furthermore, this example demonstrates that alternative conditions, such as \( u > \lambda_1 \) or \( u > \frac{r_t}{e} \), which generalize \cite[Theorem 3.7]{darwish sadek}, are insufficient to maintain the Newton polygon structure of \( g \).

\begin{example} \label{eg4}
    Consider the polynomials \( f(x) = x^5 + 4x + 4 \) and \( g(x) = x^3 + 4x + 16 \). The composition \( (g \circ f)(x) \) is given by:
    \[
        (g \circ f)(x) = x^{15} + 12x^{11} + 12x^{10} + 48x^7 + 96x^6 + 52x^5 + 64x^3 + 192x^2 + 208x + 96.
    \]
    The Newton polygons of \( g \) and \( g \circ f \) with respect to the prime \( 2 \) are:
    \begin{align*}
        NP_2(g) &: (0,0) \rightarrow (2,2) \rightarrow (3,4), \\
        NP_2(g \circ f) &: (0,0) \rightarrow (10,2) \rightarrow (14,4) \rightarrow (15,5).
    \end{align*}
    Using the notations of Theorem \ref{main thm2}, we find \( u = 2 \) and \( \beta = 5 \). Observe that \( u = 2 \) satisfies both \( u > \lambda_1 = 1 \) and \( u > \frac{r_t}{e} = \frac{4}{3} \). However, the Newton polygon structure of \( g \) is not preserved under the composition by \( f \), highlighting the necessity of the condition \( u > \lambda_t \) in Theorem \ref{main thm2}.
\end{example}

\section{Proof of Theorem \ref{p5.2}} \label{sec5}

To prove Theorem \ref{p5.2}, we first establish a technical lemma that will be instrumental in the proof.

\begin{lemma} \label{l5.1}
Let \( f(x) = a_d x^d + a_{d-1} x^{d-1} + \cdots + a_1 x + a_0 \), where \( a_0 \neq 0 \), be a polynomial of degree \( d \) with rational coefficients. Let \( p \) be a prime, and let \( g(x) = b_e x^e + b_{e-1} x^{e-1} + \cdots + b_1 x + b_0 \) be a polynomial of degree \( e \) with rational coefficients. Suppose that 
\[
\nu_p(a_0) > \max_{1 \leq j \leq e} \frac{\nu_p(b_0) - \nu_p(b_j)}{j}.
\]
Then, \( \nu_p(g(f(0))) = \nu_p(b_0) \).
\end{lemma}

\begin{proof}
From the given condition, for all \( j \) with \( 1 \leq j \leq e \), we have
\[
\nu_p(a_0) > \frac{\nu_p(b_0) - \nu_p(b_j)}{j},
\]
which implies
\[
\nu_p(b_j) + j \nu_p(a_0) > \nu_p(b_0), \quad \forall j, \ 1 \leq j \leq e.
\]
Now consider the evaluation of \( g(f(0)) \):
\[
g(f(0)) = \sum_{j=0}^e b_j a_0^j.
\]
For \( j \geq 1 \), the inequality \( \nu_p(b_j) + j \nu_p(a_0) > \nu_p(b_0) \) ensures that all other terms \( b_j a_0^j \) have strictly higher \( p \)-adic valuation than \( b_0 \). Consequently, \( \nu_p(g(f(0))) = \nu_p(b_0) \).
\end{proof}

\begin{proof}[Proof of Theorem \ref{p5.2}:]
We first establish the following assertion for all natural numbers \( n \):
\begin{align} \label{e5.1}
    f^n(x) \equiv a_d^{\frac{d^n-1}{d-1}} x^{d^n} \pmod{p} \quad \text{and} \quad \nu_p(f^n(0)) = \nu_p(a_0).
\end{align}

The assertion clearly holds for \( n=1 \) by the given conditions. Assume that the result holds for \( n = k \), i.e.,
\[
f^k(x) \equiv a_d^{\frac{d^k-1}{d-1}} x^{d^k} \pmod{p} \quad \text{and} \quad \nu_p(f^k(0)) = \nu_p(a_0).
\]
This implies
\begin{align*}
    f^{k+1}(x) &\equiv a_d^{\frac{d^k-1}{d-1}} \big(a_d x^d + a_{d-1} x^{d-1} + \cdots + a_1 x + a_0\big)^{d^k} \pmod{p} \\
    &\equiv a_d^{\frac{d^k-1}{d-1} + d^k} x^{d^{k+1}} \equiv a_d^{\frac{d^{k+1}-1}{d-1}} x^{d^{k+1}} \pmod{p},
\end{align*}
where the final congruence follows since \( \nu_p(a_i) > 0 \) for all \( i \in \{0, 1, \ldots, d-1\} \).

Next, let \( f^k(x) = \sum_{j=0}^{d^k} b_j x^j \). By the induction hypothesis on $f^k(x)$, we have \( \nu_p(b_j) \geq 1 \) for all \( j \), \( 1 \leq j \leq d^k \). Consequently, we have:
\[
\max_{1 \leq j \leq d^k} \frac{\nu_p(b_0) - \nu_p(b_j)}{j} \leq \frac{\nu_p(b_0) - 1}{j} < \nu_p(b_0) = \nu_p(f^k(0)) = \nu_p(a_0),
\]
where the final equality follows from the induction hypothesis. By applying Lemma \ref{l5.1}, it follows that
\[
\nu_p(f^{k+1}(0)) = \nu_p(b_0) = \nu_p(f^k(0)) = \nu_p(a_0).
\]
Thus, \eqref{e5.1} holds for all natural numbers \( n \).

\medskip

Since \( \nu_p(a_d) = 0 \), we have \( \nu_p(a_d^{\frac{d^n-1}{d-1}}) = 0 \), implying that the \( NP_p(f^n) \) starts at \( (0,0) \). Furthermore, \eqref{e5.1} shows that \( NP_p(f^n) \) has an endpoint at \( (d^n, \nu_p(a_0)) \), and all edges of \( NP_p(f^n) \) must have positive slopes. A chain of edges with positive increasing slopes from \( (0,0) \) to \( (d^n, \nu_p(a_0)) \) can have at most \( \nu_p(a_0) \) edges. By applying Dumas' theorem to \( NP_p(f^n) \), it follows that \( f^n(x) \) can have at most \( \nu_p(a_0) \) irreducible factors.
Therefore, \( f^n(x) \) is eventually stable.
\end{proof}

\section{Some Applications}

In this section, we give some applications of our main results, i.e., Theorems \ref{main thm1}, \ref{main thm2} and \ref{p5.2}. We start this section by proving Theorem \ref{prop1.6} which is about the families of polynomials that are dynamically irreducible at Schur polynomials.

\subsection{Proof of Theorem \ref{prop1.6}:} \label{sec6.1}

Choose a prime $p \mid m$ and fix $n \geq 0$. We show that $G_m \circ f^n$ is irreducible over $\Q$. Write $m = \sum_{i=1}^{N} b_i p^{m_i}$ with $m_1 < m_2 < \cdots < m_N$ and $0 < b_i < p$. Denote $z_j = b_1 p^{m_1} + \cdots + b_j p^{m_j}$ for $1 \leq j \leq N$. Using a method similar to that in \cite[Lemma II]{coleman} and keeping in mind that $\gcd(b_i, m) = 1$ for all $i$, $0 \leq i \leq m$, we have that $NP_p(G_m)$ has $N$ (as shown in Figure \ref{NP of Gm}) segments with slopes
\[
\lambda_i = \frac{p^{m_i} - 1}{p^{m_i}(p - 1)} < 1.
\]

\begin{figure}[htb!]
    \centering
    \begin{tikzpicture}
        \draw[<-, line width=1pt] (0,-0.5) -- (0,6);
        \draw[->, line width=1pt] (-0.5,5) -- (9.5,5);
        
        \coordinate (A) at (0,0);
        \coordinate (B) at (2,.15);
        \coordinate (C) at (3.7,0.6);
        \coordinate (D) at (5.2,1.6);
        \coordinate (E) at (6.2,2.8);
        \coordinate (F) at (7,5);
        
        \draw[line width=1.5pt] (A) -- (B) node[midway, above] {$\lambda_{1}$};
        \draw[dotted, line width=1.2pt, black!60] (B) -- (C);
        \draw[line width=1.5pt] (C) -- (D) node[midway, above] {$\lambda_{i}$};
        \draw[dotted, line width=1.2pt, black!60] (D) -- (E);
        \draw[line width=1.5pt] (E) -- (F) node[midway, right] {$\lambda_N$};
        
        \foreach \point/\position/\name in {A/below left/{$(0,-\nu_p(m!))$}, B/below/{$(z_1,-\nu_p((m-z_1)!))$}, C/right/{$(z_{i-1},-\nu_p((m-z_{i-1})!))$}, D/right/{$(z_i,-\nu_p((m-z_i)!))$}, E/left/{$(z_{N-1},-\nu_p((m-z_{N-1})!))$}, F/above/{$(m, 0)$}} {
            \fill (\point) circle (2pt) node [\position] {\name};
        }
    \end{tikzpicture}
    \caption{Newton polygon of $G_m$ with respect to $p$}
    \label{NP of Gm}
\end{figure}
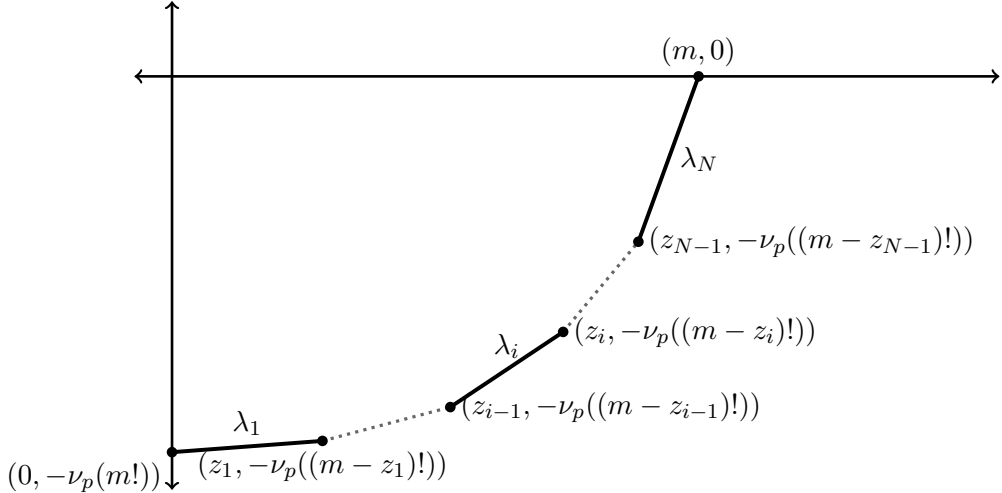
From the proof of Theorem \ref{p5.2}, we have $f^n(x) \equiv a_d^{\frac{d^n - 1}{d - 1}} x^{d^n} \pmod{p}$. Thus, by Theorem \ref{main thm2}, $NP_p(G_m \circ f^n)$ has $N$ segments with slopes
\[
\frac{p^{m_i} - 1}{d^n p^{m_i} (p - 1)}.
\]

In particular, $d^n p^{\nu_p(m)} = d^n p^{m_1}$ divides the denominator of each slope of $NP_p(G_m \circ f^n)$. Therefore, by \cite[Corollary I]{coleman}, $d^n p^{\nu_p(m)}$ divides the degree of any irreducible factor of $G_m \circ f^n$ over $\mathbb{Q}$. It follows that any irreducible factor of $G_m \circ f^n$ has degree $d^n \prod_{p\mid m} p^{\nu_p(m)} = d^n m$, which proves the result. \qed

\noindent \textbf{Remark:} To remove the restriction $\gcd(b_i, m) = 1$ for all $i$, $0 < i < m$, we can use a method similar to that in the proof of \cite[Theorem 2]{filaseta bessel}.

\subsection{Factorization of prime over a tower of number fields and it's ramification index:} \label{sec6.2-1}

For a rational prime $p$, let $\mathbb{F}_p$ denote the finite field with $p$ elements, and let $\mathbb{Z}_p$ denote the ring of $p$-adic integers. The following weaker version of the theorem of the product, originally due to Ore, relates the $\phi$-Newton polygon of a polynomial $f(x) \in \Z_p[x]$ with its factorization over $\Z_p[x]$ (\textit{cf.} \cite[Theorem 1.5]{CMS2000} for the proof).
\begin{theorem} \label{thm1.11-NP}
	Let $f(x),\phi(x) \in \Z_p[x]$ be a monic polynomials such that $\phi(x)$ is irreducible modulo a rational prime $p$. Further assume, $f(x)$ is not divisible by $\phi(x)$ but $\overline{f}(x)$ is a power of $\overline{\phi}(x)$. Suppose that the $\phi$-Newton polygon of $f(x)$ with respect to $p$ has $t$ edges $S_1,\ldots, S_t$ having slopes $\lambda_{1}< \cdots < \lambda_{t}$. Then $f(x) = f_1(x) \cdots f_t(x)$, where each $f_i(x) \in \Z_p[x]$ is a monic polynomials of degree $\ell_i \deg (\phi(x))$ and whose $\phi$-Newton polygon has a single edge, say $S_i'$, which is the translate of $S_i$ such that $\ell_i$ is the length of the horizontal projection of $S_i$.
\end{theorem}

To understand the link between the prime ideals lying above $p$ and the Newton polygon, we need the following definition of residual polynomial.

\begin{defn} \label{residual polynomial}
	Let $\phi(x) \in \mathbb{Z}_p[x]$ be a monic polynomial which is irreducible modulo a rational prime $p$ having a root $\alpha \in \overline{\mathbb{Q}}_p$. Let $f(x) \in \mathbb{Z}_p[x]$ be a monic polynomial not divisible by $\phi(x)$ whose $\phi(x)$-expansion is given by $\phi(x)^n + a_{n-1}(x)\phi(x)^{n-1} + \cdots + a_0(x)$ and such that $\overline{f}(x)$ is a power of $\overline{\phi}(x)$. 
	
	Suppose that the $\phi$-Newton polygon of $f(x)$ with respect to $p$ consists of a single edge, say $S$, having positive slope $l/e$ with $l, e$ coprime, that is,  
	\[
	\min \left\{ \frac{\nu_{p,x}(a_{n-i}(x))}{i} \ \middle| \ 1 \leq i \leq n \right\} = \frac{\nu_{p,x}(a_0(x))}{n} = \frac{l}{e},
	\]
	so that $n$ is divisible by $e$, say $n = et$, and $\nu_{p,x}(a_{n-ej}(x)) \geq lj$ with $1 \leq j \leq t$. Thus, the polynomial $b_j(x) := a_{n-ej}(x)/p^{lj}$ has coefficients in $\mathbb{Z}_p$ and $b_j(\alpha) \in \mathbb{Z}_p[\alpha]$ for $1 \leq j \leq t$. The polynomial $T(Y)$ in the indeterminate $Y$ defined by \( T(Y) = Y^t + \sum_{j=1}^{t} b_j(\alpha) Y^{t-j} \) with coefficients in $\mathbb{F}_p[\alpha] \cong \mathbb{F}_p[x]/\langle \phi(x) \rangle$ is called the residual polynomial of $f(x)$ with respect to $(\phi, S)$.
	
	Suppose that  $\phi$-Newton polygon of $f(x)$ has multiple edges say $S_1, \ldots, S_t$ having slopes $\lambda_{1}< \cdots < \lambda_{t}$. Applying Theorem \ref{thm1.11-NP}, we can write $f(x) = f_1(x) \cdots f_t(x)$, where the $\phi$-Newton polygon of $f_i(x) \in \Z_p[x]$ has a single edge say $S_i'$, which is a translate of $S_i$. The residual polynomial of $f_i(x)$ with respect to $(\phi, S_i')$ is referred to as the residual polynomial of $f(x)$ with respect to $(\phi,S_i)$. Furthermore, the polynomial $f(x)$ is said to be $p$-regular with respect to $\phi$ if none of the polynomials $T_i(Y)$ has a repeated root in the algebraic closure of of $\mathbb{F}_p$, $1 \leq i \leq t$.
\end{defn}

\begin{defn}
	Suppose $f(x) \in \Z_p[x]$ is a monic polynomial and $\overline{f}(x) = \overline{\phi_1}(x)^{e_1} \cdots \overline{\phi_r}(x)^{e_r}$ is its factorization modulo $p$ into irreducible polynomials with each $\phi_i(x) \in \Z_p[x]$ monic and $e_i>0$, then by Hensel's lemma \cite[Ch.4, Section 3]{ShBo66}, there exists monic polynomials $f_1(x), \ldots, f_r(x)$ in $\Z_p[x]$ such that $f(x) = f_1(x) \cdots f_r(x)$ and $\overline{f_i}(x) = \overline{\phi_i}(x)^{e_i}$ for each $i$. The polynomial $f(x)$ is said to be $p$-regular (with respect ito $\phi_1, \ldots, \phi_r$) if each $f_i(x)$ is $p$-regular with repect to $\phi_i$. 
\end{defn}
\noindent To determine the number of distinct prime ideals of $\mathbb{Z}_K$ lying above a rational prime $p$, we will use the following theorem which is a weaker version of \cite[Theorem 1.2]{KK12}.

\begin{theorem} \label{prime factorization} 
	Let $K = \mathbb{Q}(\theta)$ be a number field with $\theta$ satisfying an irreducible polynomial $f(x) \in \mathbb{Z}[x]$ and $p$ be a rational prime. Let $\phi_1(x)^{e_1} \cdots \phi_r(x)^{e_r}$ be the factorisation of $f(x)$ modulo $p$ into powers of distinct irreducible polynomials over $\mathbb{F}_p$ with each $\phi_i(x)$ belonging to $\mathbb{Z}[x]$ monic and not dividing $f(x)$. For each $i$, suppose that the $\phi_i$-Newton polygon of $g(x)$ have $t_i$ edges $S_{ij}$ with slopes $\lambda_{ij} = l_{ij}/e_{ij}$, where $\gcd(l_{ij}, e_{ij}) = 1$. If $f(x)$ is $p$-regular and $T_{ij}(Y) = \prod_{s=1}^{s_{ij}} U_{ijs}(Y)$ is the factorisation of the residual polynomial $T_{ij}(Y)$ into distinct irreducible factors over $\mathbb{F}_p$ with respect to $(\phi_i, S_{ij})$ for $1 \leq j \leq t_i$, then  
	\[
	p \mathbb{Z}_K = \prod_{i=1}^{r} \prod_{j=1}^{t_i} \prod_{s=1}^{s_{ij}} \mathfrak{p}_{ijs}^{e_{ij}},
	\]
	where $\mathfrak{p}_{ijs}$ are distinct prime ideals of $\mathbb{Z}_K$ having residual degree $\deg \phi_i(x) \cdot \deg U_{ijs}(Y)$.  
\end{theorem}

For a given prime $p$, we highlight here that using Theorems \ref{main thm1} and \ref{main thm2}, one can easily construct a tower of number fields $K_n$ generated by roots of irreducible polynomials $f^n(x)$ for $n \in \mathbb{N}$. Further, it will also give the factorization of $p\Z_{K_n}$ over $K_n$. 

For a prime $q\neq p$. Consider the polynomial $f(x) = x^d + q^{d-1} a x^m + q^{d-1} b \in \Q[x]$ with $\nu_q(a)= \nu_q(b)= 0 $. Clearly, $f(x)$ is $q^{d-1}$-Dumas polynomial. Applying Corollary \ref{cor6.2}, we deduce that all the iterates $f^n$ of $f$ are $q^{d-1}$-Dumas and hence, irreducible.

Select $a,b,m$ such that $\gcd(d-m, \nu_p(a))= \gcd(m, \nu_p(b) - \nu_p(a))=1$. Suppose that the vertices of $NP_p(f)$ are given by the set $S=\{ (0,0), (d-m, \nu_p(a)), (d, \nu_p(b)) \}$. Define the slopes:
\[ \lambda_{1} = \frac{\nu_p(a)}{d-m}, \quad \lambda_{2} = \frac{\nu_p(b) - \nu_p(a)}{m}.\]
Refine the choices of $a,b$ to satisfy $p \mid a$, $p \mid b$ and $\nu_p(b) \leq \lambda_{1} (2d-1)$. Applying Theorem \ref{main thm1}, we observe that for all $n \in \N$, the structure of Newton polygon of $f(x)$ with respect to $p$ is given by
\[NP_p(f) : (0,0) \longrightarrow (d^{n-1}(d-m), \nu_p(a)) \longrightarrow (d^n, \nu_p(b)). \]
Assumption $\gcd(d-m, \nu_p(a))= \gcd(m, \nu_p(b) - \nu_p(a))=1$ ensures that the residual polynomials corresponding to each edge are linear. For a root $\theta_n$ of $f^n$, let $K_n$ be the field generated by $\theta_n$ over $\Q$. Applying Theorem \ref{prime factorization}, we get
$$p\Z_{K_n} = \mathfrak{p_1}^{d-m} \mathfrak{p_2}^m $$
where $\mathfrak{p_1}$ and $\mathfrak{p_2}$ are distinct prime ideals of $\Z_{K_n}$ having residual degree $1$. If $d \geq 3$, then for all $n \geq 1$, $p$ ramifies in $K_n$.

\subsection{Non-monogenity of number fields:} \label{sec6.2}

Let $K = \mathbb{Q}(\theta)$ be an algebraic number field with $\theta$ in the ring $\mathbb{Z}_K$ of algebraic integers of $K$. Let $f(x)$ be the minimal polynomial of $\theta$ having degree $n$ over the field $\mathbb{Q}$ of rational numbers. A number field $K$ is said to be \emph{monogenic} if there exists some $\alpha \in \mathbb{Z}_K$ such that $\mathbb{Z}_K = \mathbb{Z}[\alpha]$. In this case, $\{1, \alpha, \ldots, \alpha^{n-1}\}$ is an integral basis of $K$; such an integral basis of $K$ is called a \emph{power integral basis}, or briefly, a \emph{power basis} of $K$. If $K$ does not admit any power basis, it is said to be \emph{non-monogenic}.

Let $\operatorname{ind} \theta$ denote the index of the subgroup $\mathbb{Z}[\theta]$ in $\mathbb{Z}_K$, and let $i(K)$ represent the \emph{index of the field} $K$, defined as 
\[
i(K) = \gcd\{\operatorname{ind} \alpha \mid K = \mathbb{Q}(\alpha) \text{ and } \alpha \in \mathbb{Z}_K\}.
\]
A prime number $p$ dividing $i(K)$ is called a \emph{prime common index divisor} of $K$. Note that if $K$ is monogenic, then $i(K) = 1$. Consequently, a number field possessing a prime common index divisor cannot be monogenic.

For a rational prime $p$, let $\mathbb{F}_p$ denote the finite field with $p$ elements, and let $\mathbb{Z}_p$ denote the ring of $p$-adic integers. The following result from \cite[Theorem 4.34]{Nar04} will be utilized to prove the nonmonogenity of number field $K$.

\begin{lemma} \label{4.1}
	Let $K$ be an algebraic number field and $p$ be a rational prime. Let $P_h$ denote the number of distinct prime ideals of $\mathbb{Z}_K$ lying above $p$ having residual degree $h$, and let $N_h$ denote the number of irreducible polynomials of degree $h$ in $\mathbb{F}_p[x]$. Then $p$ is a prime common index divisor of $K$ if and only if $P_h > N_h$ for some $h$.
\end{lemma}

Recently, using the aforementioned result, many authors have constructed classes of non-monogenic number fields generated by roots of binomials, trinomials, quadrinomials, and other types of polynomials. For further details, we refer the reader to the recent survey article \cite{Gaal}.

We highlight here that using Theorems \ref{main thm1} and \ref{main thm2} with Lemma \ref{4.1}, one can easily deduce a tower of non-monogenic number fields $K_n$ generated by roots of irreducible polynomials $f^n(x)$ for $n \in \mathbb{N}$. For example:

Let $f(x) = x^d + 3^{d-1} a x^m + 3^{d-1} b x^l + 3^{d-1} c \in \mathbb{Q}[x]$ with $\nu_3(a) = \nu_3(b) = \nu_3(c) = 0$. Suppose that the vertices of the Newton polygon $NP_2(f)$ are given by the set 
\[
S = \{(0, 0), (d - m, \nu_2(a)), (d - l, \nu_2(b)), (d, \nu_2(c))\}.
\]
Assuming $\gcd(d - m, \nu_2(a))= \gcd(m - l, \nu_2(b) - \nu_2(a))= \gcd(l, \nu_2(c) - \nu_2(b))=1$, define the slopes:
\[
\lambda_1 = \frac{\nu_2(a)}{d - m}, \quad 
\lambda_2 = \frac{\nu_2(b) - \nu_2(a)}{m - l}, \quad 
\lambda_3 = \frac{\nu_2(c) - \nu_2(b)}{l}.
\]

Clearly, $f(x)$ is a $3^{d-1}$-Dumas polynomial. By Corollary \ref{cor6.2}, all iterates $f^n$ of $f$ are $3^{d-1}$-Dumas and hence irreducible. Further, choose $a, b, c \in \mathbb{Q}$ such that $\nu_2(c) \leq \lambda_1(2d - 1)$.

Under this assumption, Theorem \ref{main thm1} implies that for all $n \in \mathbb{N}$, the Newton polygon of $f^n$ with respect to 2 is given by:
\[
NP_2(f^n): (0, 0) \to (d^{n-1}(d - m), \nu_2(a)) \to (d^{n-1}(d - l), \nu_2(b)) \to (d^n, \nu_2(c)).
\]
The constraints on the $2$-adic valuations of $a, b, c \in \mathbb{Q}$ ensures that the residual polynomials (\textit{cf.} Definition \ref{residual polynomial}) 
corresponding to each edge are linear. If $\theta_n$ is a root of $f^n$ and $K_n = \mathbb{Q}(\theta_n)$, then by Theorem \ref{prime factorization}, we see that the number of distinct prime ideals of $\Z_{K_n}$ lying above $2$ having residual degree 1 are $3$. However, we know that the number of distinct irreducible linear polynomials in $\mathbb{F}_2[x]$ are $2$. Therefore, by Lemma \ref{4.1}, $2 \mid i(K_n)$ for all $n \in \mathbb{N}$. This establishes that for all $n \in \mathbb{N}$, $f^n$ is non-monogenic.

There exist infinitely many such families of polynomials. One particular family is 
\[
f(x) = x^4 + 2 \cdot 3^3 a x^3 + 16 \cdot 3^3 b x + 128 \cdot 3^3 c,
\]
where $a, b, c \in \mathbb{Q}$ satisfy $\nu_q(a) = \nu_q(b) = \nu_q(c) = 0$ for $q = 2, 3$. Clearly, $f(x)$ is a $3^3$-Dumas polynomial. Further, all the iterates $f^n$ generates tower of non-monogenic number fields $K_n$.

\subsection{Number of Irreducible Factors, Eventual Stability, and Degree of Factors:} \label{sec6.3}

Let $g(x) \in \mathbb{Q}[x]$ be a polynomial of degree $e$, and let $p$ be a prime such that the Newton polygon of $g$ with respect to $p$ is given by:
\[
NP_p(g): (0, 0) \rightarrow (e - m_1, r_1) \rightarrow \cdots \rightarrow (e - m_s, r_s) \rightarrow \cdots \rightarrow (e, r_t).
\]
Here, we define $m_0 = e$, $m_t = r_0 = 0$, and let $0 < \lambda_1 < \lambda_2 < \cdots < \lambda_t$ be the slopes of the edges of $NP_p(g)$. If $r_t \leq \lambda_1(2e - 1)$, then Theorem \ref{main thm1} implies that, for any $n \in \mathbb{N}$, the Newton polygon of the $n$th iterate of $g$, denoted by $g^n$, is given by:
\[
NP_p(g^n): (0, 0) \rightarrow (e^{n-1}(e - m_1), r_1) \rightarrow \cdots \rightarrow (e^{n-1}(e - m_s), r_s) \rightarrow \cdots \rightarrow (e^n, r_t).
\]

In this case, the number of irreducible factors of $g^n$ is bounded by $r_t$, ensuring that $g^n$ is \emph{eventually stable}. Furthermore, if we assume 
\[
\gcd(e(m_{s+1} - m_s), r_{s+1} - r_s) = 1 \quad \text{for all } 0 \leq s < t,
\]
then the vertices of the Newton polygon are the only lattice points on $NP_p(g)$. Applying Dumas' Theorem, it follows that $g^n$ remains eventually stable, with the number of irreducible factors of $g^n$ bounded by $t$. 

Additionally, Dumas' Theorem implies that the degree of any irreducible factor of $g^n$ belongs to the set:
\[
\left\{ \sum_{j=1}^{k} (m_{i_j+1} - m_{i_j}) : 1 \leq k \leq n \text{ and } 1 \leq i_1, i_2, \ldots, i_k \leq n \right\}.
\]

{}

\subsection{Conjecture of Sookdeo:} \label{sec6.9}

Let $K$ be a number field and $f(x) \in K(x)$ be a non-constant rational function. The \textit{forward orbit} of $\alpha \in \mathbb{P}^1(K)$ under $f$ is defined as 
\[
O_f^+(\alpha) = \{ \alpha, f(\alpha), f^2(\alpha), \ldots \},
\]
and the \textit{backward orbit} of $\alpha$, denoted by $O_f^-(\alpha)$, is defined as 
\[
O_f^-(\alpha) = \bigcup_{n \geq 0} f^{-n}(\alpha),
\]
where 
\[
f^{-n}(\alpha) := \{ \beta \in \mathbb{P}^1(\overline{K}) : f^n(\beta) = \alpha \}.
\]
A point $\alpha$ is said to be \textit{preperiodic} for $f$ if its forward orbit $O_f^+(\alpha)$ is finite. Similarly, $\alpha$ is \textit{exceptional} for $f$ if its backward orbit $O_f^-(\alpha)$ is finite. 

Consider a point $\alpha \in \mathbb{P}^1(K)$. For each $n \geq 1$, choose coprime $a_n, b_n \in K[x]$ with $f^n(z) = a_n(x) / b_n(x)$. If $\alpha \neq \infty$, we say that the pair $(f, \alpha)$ is \textit{eventually stable}, if the number of irreducible factors of $a_n(x) - \alpha b_n(x)$ in $K[x]$ is bounded by a constant independent of $n$. We say that $(f, \infty)$ is \textit{eventually stable}, if the number of irreducible factors of $b_n(z)$ is similarly bounded.

Let $S$ be a finite set of places of $K$ containing all Archimedean places. A point $\beta \in \mathbb{P}^1(\overline{K})$ is called \textit{$S$-integral with respect to $\gamma \in \mathbb{P}^1(K)$} if there is no prime $\mathfrak{p}$ of $K(\beta)$ lying over a prime outside of $S$ such that the images of $\beta$ and $\gamma$ modulo $\mathfrak{p}$ coincide. Define the set
\[
\mathcal{O}_{S, \gamma} := \{ \beta \in \mathbb{P}^1(\overline{K}) : \beta \text{ is $S$-integral with respect to } \gamma \}.
\]
If $S$ consists only of the Archimedean places of $K$, then $\mathcal{O}_{S, \infty}$ is the ring of algebraic integers in $\overline{K}$. 

With the above definitions, Silverman \cite{Sil93} in 1993 proved that: For a rational function $f(x) \in \mathbb{Q}(x)$ of degree at least $2$ and $\alpha \in \mathbb{P}^1(K)$, if $\infty$ is not exceptional for $f$, then the forward orbit $O_f^+(\alpha)$ contains only finitely many points in $\mathbb{P}^1(K)$ that are $S$-integral relative to $\infty$. 

In 2011, Sookdeo \cite[Theorem 2.5 and 2.6]{sook11} proved the following analogue of Silverman's result for backward orbits $\mathcal{O}_{f}^-(\a)$.

\begin{theorem} \label{t6.4}
Let $K$ be a number field, $S$ a finite set of places of $K$ containing all Archimedean places, $\alpha \in \mathbb{P}^1(K)$, and $f(x) \in K(x)$ be a rational function of degree $d \geq 2$. If $(f, \alpha)$ is eventually stable, then
\[
\mathcal{O}_{S, \gamma} \cap O_f^-(\alpha) \text{ is finite for all } \gamma \in \mathbb{P}^1(K) \text{ not preperiodic under } f.
\]
\end{theorem}

From Theorem \ref{t6.4}, we observe that \emph{eventual stability} plays a crucial role in proving the Conjecture of Sookdeo. Using Theorems \ref{main thm1} and \ref{main thm2}, one can easily generate infinitely many families of polynomials in $\mathbb{Q}[x]$ satisfying the conjecture. One such family is described below.

For an integer $n \geq 3$, consider the polynomial 
\[
g(x) = x^{3n} + 2a x^{2n-2} + 4b x^{n-2} + 8c,
\]
where $a$ and $b$ are positive rationals with $\nu_2(a) = \nu_2(b) = 0$, and $c$ is an odd positive integer. A simple application of Theorem \ref{main thm1} shows that $g^n(x)$ is eventually stable. Moreover, it can be easily verified that
\[
8 \leq g(0) < g^2(0) < \cdots < g^n(0) < \cdots,
\]
which implies that $0$ is not preperiodic under $g$. Applying Theorem \ref{t6.4}, it follows that Sookdeo's conjecture \cite[Conjecture 1.2]{sook11} holds for the polynomial $g(x)$.

\section*{Declarations}

\noindent \textbf{Data Availability:} Data sharing is not applicable to this article as no datasets were generated or analysed during the current study.\\

\noindent \textbf{Conflict of interest:} On behalf of all authors, the corresponding author states that there is no conflict of interest.

\end{document}